\documentclass[11pt, a4paper]{amsart}
\usepackage{amsmath}
\usepackage{commath}
\usepackage{amsthm}
\usepackage{natbib}
\usepackage{amsfonts}
\usepackage[margin=1.25in]{geometry}
\usepackage{booktabs}
\usepackage{graphicx}
\usepackage{lscape}
\usepackage{setspace}
\usepackage{microtype}
\usepackage{rotating}
\newcommand{\sgn}{\mathrm{sign}}

\newtheorem{theorem}{Theorem}
\newtheorem{lemma}{Lemma}
\newtheorem{corollary}{Corollary}
\newtheorem{definition}{Definition}
\newtheorem{assumption}{Assumption}
\begin{document}

\title[Oracle Inequalities for Vector Autoregressions]{Oracle Inequalities for High Dimensional Vector Autoregressions}
\author{Anders Bredahl Kock and Laurent A.F. Callot}
\date{\today}
\thanks{We would like to thank Michael Jansson, S\o{}ren Johansen, J\o{}rgen Hoffmann-J\o{}rgensen, Marcelo Medeiros, Martin Wainwright, Timo Ter\"asvirta, participants at the workshop for statistical inference in complex/high-dimensional problems held in Vienna, the NBER-NSF 2012 time series conference and the Canadian Econometric Study Group meeting 2012, the SETA 2013, for help, comments and discussions. Financial support from the Center for Research in the Econometric Analysis of Time Series (CREATES) is gratefully acknowledged.\\
Anders Bredahl Kock is affiliated with Aarhus University and CREATES, Fuglesangs Alle 4, 8210 Aarhus V, Denmark. Laurent Callot is affiliated with VU Amsterdam, CREATES, and the Tinbergen Institute. Corresponding author: Anders Bredahl Kock, e-mail: \texttt{akock@creates.au.dk}.}
\begin{abstract}
This paper establishes non-asymptotic oracle inequalities for the prediction error and estimation accuracy of the LASSO in stationary vector autoregressive models. These inequalities are used to establish consistency of the LASSO even when the number of parameters is of a much larger order of magnitude than the sample size. We also give conditions under which no relevant variables are excluded.

Next, non-asymptotic probabilities are given for the adaptive LASSO to select the correct sparsity pattern. We then give conditions under which the adaptive LASSO reveals the correct sparsity pattern asymptotically. We establish that the estimates of the non-zero coefficients are asymptotically equivalent to the oracle assisted least squares estimator. This is used to show that the rate of convergence of the estimates of the non-zero coefficients is identical to the one of least squares only including the relevant covariates.

\vspace{.5cm}
\footnotesize
{

}
\end{abstract}

\maketitle

\section{Introduction}

The last 10-15 years have witnessed a surge of research in high-dimensional statistics and econometrics. This is the study of models where the the number of parameters is of a much larger order of magnitude than the sample size. However, often only a few of the parameters are non-zero, that is the model is sparse, and one wants to be able to separate these from the zero ones. In particular, a lot of attention has been devoted to penalized estimators of which the most famous is probably the LASSO of \cite{tibshirani96}. Other prominent examples are the SCAD of \cite{fanl01}, the adaptive LASSO of \cite{zou06}, the Bridge and Marginal Bridge estimators of \cite{huanghm08}, the Dantzig selector of \cite{candest07}, the Sure Independence Screening of \cite{fanl08} and the square root LASSO of \cite{bellonicw11}. These procedures have become popular since they are computationally feasible and perform variable selection and parameter estimation at the same time. For recent reviews see \cite{bühlmannvdg11} and \cite{bellonic11}. 

Much effort has been devoted to establishing the conditions under which these procedures possess the oracle property. Here the oracle property is understood as the procedure correctly detecting the sparsity pattern, i.e. setting all zero parameters \textit{exactly} equal to zero while not doing so for any of the non-zero ones. Furthermore, the non-zero parameters are estimated at the same asymptotic efficiency as if only the relevant variables had been included in the model from the outset. In other words the non-zero parameters are estimated as efficiently as if one had been assisted by an oracle that had revealed the true sparsity pattern prior to estimation.

Even though a lot of progress has been made in this direction most focus has been devoted to very simple data types such as the linear regression model with fixed covariates or sometimes (gaussian) independently distributed covariates. Some exceptions are \cite{wanglt07} and \cite{nardir11} who consider the LASSO in a stationary autoregression. \cite{canerk13} investigate the properties of the Bridge estimator in stationary and nonstationary autoregressions while \cite{kock12} does the same for the adaptive LASSO. Finally, \cite{liaop13} consider cointegrated VAR models. In particular they show how shrinkage type estimators can also be used to select the cointegration rank. However, these papers consider (vector) autoregressions of a fixed or slowly increasing length -- i.e. a low-dimensional setting.
  
In this paper we are concerned with the estimation of high-dimensional stationary vector autoregressions (VAR), i.e. models of the form
\begin{align}
y_t=\sum_{l=1}^{p_T}\Phi_ly_{t-l}+\epsilon_t,\ t=1,...,T\label{VAR1}
\end{align}
where $y_t=(y_{t,1}, y_{t,2},...,y_{t,k_T})'$ is the $k_T\times 1$ vector of variables in the model. $\Phi_1,...,\Phi_{p_T}$ are $k_T\times k_T$ parameter matrices. These may vary with $T$ though we suppress this dependence in the notation. We are analyzing a triangular array of models where the parameters may vary across the rows, $T$, but remain constant within each row, $t=1,...,T$. $\epsilon_t$ is assumed to be a sequence of $i.i.d.$ error terms with an $N_{k_T}(0,\Sigma)$ distribution. Notice that the number of variables as well as the number of lags is indexed by $T$ indicating that both of these are allowed to increase as the sample size increases -- and in particular may be larger than $T$. Equation (\ref{VAR1}) could easily be augmented by a vector of constants but here we omit this to keep the notation simple\footnote{Similarly, we conjecture that a trend could be included by writing the model in deviations from the trend. But to focus on the main idea of the results this has been omitted.}.

The VAR is without doubt one of the central pillars in macroeconometrics and is widely used for e.g. forecasting, impulse response and policy analysis. However, it suffers from the fact that many macroeconomic variables are observed at a relatively low frequency such as quarterly or annually leaving few observations for estimation. On the other hand, the number of parameters, $k_T^2p_T$, may increase very fast if many variables are included in the model which is often the case in order to ensure satisfactory modeling of the dynamics of the variables of interest. Hence, the applied researcher may find himself in a situation where the number of parameters is much larger than the number of observations. If $T<k_Tp_T$ equation by equation least squares is not even feasible since the design is singular by construction. Even if it is possible to estimate the model the number of regressions which have to be run in order to calculate the information criterion for every subset of variables increases exponentially in the number of parameters and such a procedure would become computationally infeasible. Furthermore, these subset selection criteria are known to be unstable; see e.g. \cite{breiman96}.

In a seminal paper \cite{stockw02} used factors to reduce dimensionality and obtain more precise forecasts of macro variables while \cite{bernankebe05} popularized the inclusion of factors in the VAR in order to avoid leaving out relevant information when evaluating monetary policy. For surveys on factor analysis in the context of time series see \cite{stockw06}, \cite{baing08} and \cite{stockw11}.
Our results open a different avenue of handling high-dimensionality in VAR models than augmentation by factors. 
In particular,
\begin{itemize}
\item[i)] we establish \textit{non-asymptotic} oracle inequalities for the prediction error and estimation accuracy of the LASSO. Specifically,  we show that the LASSO is almost as precise as oracle assisted least squares, i.e. least squares only including the (unknown set of) relevant variables. We also comment on some limitations of these bounds.
\item[ii)] we use the finite sample upper bounds to show that even when $k_T$ and $p_T$ increase at a subexponential rate it is possible to estimate the parameters consistently. The fact that $k_T$ may increase very fast is of particular importance for state of the art macroeconometric modeling of big systems since no variables should be left out in order to avoid omitted variable bias. 
Conditions for no relevant variables being excluded are also given.
\item[iii)] we establish \textit{non-asymptotic} lower bounds on the probability with which the adaptive LASSO unveils the correct sign pattern and use these bounds to show that the adaptive LASSO may detect the correct sign pattern with probability tending to one even when $k_T$ and $p_T$ increase at a subexponential rate\footnote{Increasing $p_T$ by one implies losing one observation for the purpose of estimation since one more initial observation is needed. Hence, we also discuss the practically relevant setting where $p_T$ increases slower than the sample size (as opposed to exponentially fast).}.
\item[iv)] we show that the adaptive LASSO is asymptotically equivalent to the oracle assisted least squares estimator. This implies that the estimates of the non-zero coefficients converge at the same rate as if least squares had been applied to a model only including the relevant covariates. Furthermore, this shows that the adaptive LASSO is asymptotically as efficient as the oracle assisted least squares estimator.
\item[v)] the appendix contains some maximal inequalities for vector autoregressions, Lemmas \ref{Bound} and \ref{BoundY}, which might be of independent interest. In addition, Lemma \ref{CovBound} gives finite sample bounds on the estimation error of high-dimensional covariance matrices in VAR models. This concentration inequality is subsequently used to show how the restricted eigenvalue condition of \cite{bickelrt09} can be verified even in the case of models with many \textit{dependent} random covariates.
\item[vi)] similar results for autoregressions follow as a special case by simply setting $k_T=1$ in our theorems. 
\item[vii)] in \cite{kockc12oracle} we show how the theory put forward in this paper can be used to analyze and forecast a large macroeconomic data set; using the 131 macroeconomic variables of \cite{ludvigson2009macro} we show that the LASSO provides forecasts that compare favourably to those obtained using alternative methods.
\end{itemize}
We believe that these results will be of much use for the applied researcher  who often faces the curse of dimensionality when building VAR models since the number of parameters increases quadratically in the number of variables included. The LASSO and the adaptive LASSO are shown to have attractive finite sample and asymptotic properties even in these situations.

\cite{Songb11} have derived bounds for the LASSO similar to the ones mentioned in i) above. However, they employ an m-dependence type assumption instead of directly utilizing the dependence structure in the VAR. Furthermore, in our understanding they are working with models of fixed dimension while we are allowing $k_T$ and $p_T$ to depend on the sample size. Also, we give a lower bound on the probability with which the restricted eigenvalue condition is satisfied. Finally, we also consider the properties of the adaptive LASSO non-asymptotically as well as asymptotically.

Note that since the LASSO and the adaptive LASSO can be estimated with fewer observations than parameters one may choose to simply include the most recent observations -- say 10-20 years -- in the model used for forecasting instead of using the whole data set. This could be useful since observations far back in time may be conjectured to be less informative about the near future than the recent past is. 

It should also be remarked that no (significance) testing is involved in the procedures but the underlying assumption is that there exists a sparse representation of the data. Finally, by iii) above, the adaptive Lasso may be used to unveil Granger (non)-causality since it can distinguish between zero and non-zero coefficients under the conditions of our theorems.

The plan of the paper is as follows. Section \ref{ModelNotation} lays out the model in more detail and gives necessary background notation. Sections \ref{SecLASSO} and \ref{SecAdaLASSO} contain the main results of the paper on the LASSO and the adaptive LASSO.
A Monte Carlo study investigating the validity of our finite sample results can be found in Section \ref{MC} while Section \ref{Conclusions} concludes. The proofs can be found in the Appendix.

\section{Model and notation}\label{ModelNotation}
We shall suppress the dependence of $k_T$ and $p_T$ on $T$ to simplify notation. Throughout the paper we shall assume
\begin{assumption}\label{Ass1}
The data $\cbr[0]{y_{t},\ t=1-p,...,T}$ is generated by (\ref{VAR1}) where $\epsilon_t$ is assumed to be a sequence of $i.i.d.$ error terms with an $N_{k}(0,\Sigma)$ distribution. Furthermore, all roots of $\envert[0]{I_k-\sum_{j=1}^p\Phi_jz^j}$ are assumed to lie outside the unit disc. 
\end{assumption}
Assumption \ref{Ass1} states the model. The gaussianity of the error terms is crucial since $y_t$ inherits this under the assumption of all roots of $\envert[0]{I_k-\sum_{j=1}^p\Phi_jz^j}$ being outside the unit disc. More precisely, the gaussianity is useful since it implies that $y_t$ and $\epsilon_t$ have slim tails\footnote{Note that \cite{belloni12sparse} derive bounds similar to ours in an IV setting without imposing gaussianity on the error terms. Instead they use moderate deviation theory. However, we have sofar been unable to apply this theory in our setting of dependent variables.}. Hence, the gaussianity of the error terms plays an important role. It is, however, a rather standard assumption. Note also that the assumption of all roots of $\envert[0]{I_k-\sum_{j=1}^p\Phi_jz^j}$ being outside the unit disc is equivalent to all eigenvalues of the companion matrix $F$ being inside the unit disc. Let $\rho$ (the dependence on $T$ is suppressed) denote the largest eigenvalue of the companion matrix. 

It is convenient to write the model in stacked form. To do so let $Z_{t}=(y_{t-1}',...,y_{t-p}')'$ be the $kp\times 1$ vector of explanatory variables at time $t$ in each equation and $X=(Z_{T},...,Z_{1})'$ the $T\times kp$ matrix of covariates for each equation. Let $y_i=(y_{T,i},...,y_{1,i})'$ be the $T\times 1$ vector of observations on the $i$th variable ($i=1,...,k$) and $\epsilon_i=(\epsilon_{T,i},...,\epsilon_{1,i})'$ the corresponding vector of error terms. Finally, $\beta_{i}^*$ is the $kp$ dimensional parameter vector of true parameters for equation $i$ which also implicitly depends on $T$. Hence, we may write (\ref{VAR1}) equivalently as
\begin{align}
y_i=X\beta_{i}^*+\epsilon_i,\ i=1,...,k.\label{VAR2}
\end{align}
Here the length of the parameter vector $\beta_{i}^*$ can potentially be of a much larger order of magnitude than the sample size $T$. A practical example occurs when building macroeconomic models based on relatively infrequently sampled time series (say quarterly or annual data). Then one will often only have 50-200 observations while for $k=100$ and $p=5$ the number of parameters per equation is as large as 500. The total number of parameters in the system is of course even larger. Traditional methods such as least squares will be inadequate in such a situation and our goal is to derive bounds on the estimation error of the LASSO in each equation as well as for the whole system.

Even though there are $k^2p$ parameters in the model only a subset of them might be non-zero. Perhaps only a few of the $kp$ variables might be relevant in each of the $k$ equations (the set may of course differ from equation to equation such that all variables are relevant in some equations). This means that $\beta_{i}^*$ is a sparse vector. 

\subsection{Further notation}
Let $J_{i}=\cbr[0]{j: \beta_{i,j}^*\neq 0}\subseteq \cbr[0]{1,...,kp}$ denote the set of non-zero parameters in equation $i$ and $s_i=|J_i|$ its cardinality. $\bar{s}=\max\cbr[0]{s_1,...,s_k}$ while $\beta_{\min,i}=\min\cbr[0]{|\beta_{i,j}^*|:j\in J_{i}}$ denotes the minimum (in absolute value) non-zero entry of $\beta_{i}^*$. $\beta_{\min}=\min\cbr[0]{\beta_{\min,i},\ i=1,...,k}$ is the smallest non-zero parameter in the whole system. 
     
For any $x\in \mathbb{R}^n$, $\enVert[0]{x}=\sqrt{\sum_{i=1}^nx_i^2}$, $\enVert[0]{x}_{\ell_1}=\sum_{i=1}^n |x_i|$ and $\enVert[0]{x}_{\ell_\infty} =\max_{1\leq i\leq n} |x_i|$ denote $\ell_2,\ \ell_1$ and $\ell_{\infty}$ norms, respectively (most often $n=kp$ or $n=s_{i}$ in the sequel). For any symmetric square matrix $M$, $\phi_{\min}(M)$ and $\phi_{\max}(M)$ denote the minimal and maximal eigenvalues of $M$.

Let $\Psi_T=\frac{1}{T}X'X$ be the $kp\times kp$ scaled Gramian of X. For $R,S\subseteq \cbr[0]{1,...,kp}$, $X_R$ and $X_S$ denote the submatrices of $X$ which consist of the columns of $X$ indexed by $R$ and $S$, respectively. Furthermore, $\Psi_{R,S}=\frac{1}{T}X_R'X_S$. For any vector $\delta$ in $\mathbb{R}^n$ and a subset $J\subseteq \left\{1,...,n\right\}$ we shall let $\delta_J$ denote the vector consisting only of those elements of $\delta$ indexed by $J$. As mentioned above, $F$ shall denote the companion matrix of the system (\ref{VAR1}) and $\rho$ its largest eigenvalue. For any natural number $j$, $F^j$ shall denote the $j$th power of $F$.

For any two real numbers $a$ and $b$, $a\vee b=\max(a,b)$ and $a\wedge b=\min(a,b)$; and for any $x\in\mathbb{R}^n$, let $\sgn(x)$ denote the sign function applied to each component of $x$.

Let $\sigma_{i,y}^2$ denote the variance of $y_{t,i}$ and $\sigma_{i,\epsilon}^2$ the variance of $\epsilon_{t,i},\ 1\leq i\leq k$. Then define $\sigma_T=\max_{1\leq i \leq k}(\sigma_{i,y}\vee \sigma_{i,\epsilon})$. 

\section{The Lasso}\label{SecLASSO}
The LASSO was proposed by \cite{tibshirani96}. Its theoretical properties have been studied extensively since then, see e.g. \cite{zhaoy06}, \cite{meinshausenb06}, \cite{bickelrt09}, and \cite{bühlmannvdg11} to mention just a few. It is known that it only selects the correct model asymptotically under rather restrictive conditions on the dependence structure of the covariates. However, we shall see that it can still serve as an effective screening device in these situations. Put differently, it can remove many irrelevant covariates while still maintaining the relevant ones and estimating the coefficients of these with high precision. We investigate the properties of the LASSO when applied to each equation $i=1,...,k$ separately\footnote{Of course it is also possible to apply the LASSO directly to the whole system and we shall also make some comments on the properties of the resulting estimator under this strategy.}. The LASSO estimates $\beta_{i}^*$ in (\ref{VAR2}) by minimizing the following objective function

\begin{align}
L(\beta_{i})=\frac{1}{T}\enVert{y_{i}-X\beta_{i}}^2+2\lambda_T\enVert{\beta_{i}}_{\ell_1}\label{LASSOobj}
\end{align}
where $\lambda_T$ is a sequence to be defined exactly below. (\ref{LASSOobj}) is the least squares objective function plus an extra term penalizing parameters that are different from zero. Let $\hat{\beta}_{i}$ denote the minimizer of (\ref{LASSOobj}) and let $J(\hat{\beta}_{i})=\cbr[0]{j:\hat{\beta}_{i,j}\neq 0}$ be the indices of the parameters for which the estimator is nonzero . 

\subsection{Results without conditions on the Gram matrix}
We begin by giving \textit{non-asymptotic} bounds on the performance of the LASSO. Notice that these bounds are valid under assumption \ref{Ass1} without any conditions on the design matrix or $k,p,s_{i}$ and $T$. 

\begin{theorem}\label{Thm1}
Let $\lambda_T=\sqrt{8\ln(1+T)^5\ln(1+k)^4\ln(1+p)^2\ln(k^2p)\sigma_T^4/T}$. Then, on a set with probability at least $1-2(k^2p)^{1-\ln(1+T)}-2(1+T)^{-1/A}$ the following inequalities hold for all $i=1,...,k$ for some positive constant $A$.\footnote{At the cost of a slightly more involved expression on the lower bound on the probability with which the expressions hold, $\lambda_T$ may be reduced to $\sqrt{8\ln(1+T)^{3+\delta}\ln(1+k)^4\ln(1+p)^2\ln(k^2p)\sigma_T^2/T}$ for any $\delta>0$. This remark is equally valid for all theorems in the sequel.\label{smalllambda}}
\begin{align}
&\frac{1}{T}\enVert[1]{X\hat{\beta_{i}}-X\beta_{i}^*}^2+\lambda_T\enVert[1]{\hat{\beta_{i}}-\beta_{i}^*}_{\ell_1}\leq 2\lambda_T\left(\enVert[1]{\hat{\beta_{i}}-\beta_{i}^*}_{\ell_1}+\enVert[1]{\beta_{i}^*}_{\ell_1}-\enVert[1]{\hat{\beta}_{i}}_{\ell_1}\right)\label{IQ1}\\
&\frac{1}{T}\enVert[1]{X\hat{\beta}_{i}-X\beta^*_i}^2+\lambda_T\enVert[1]{\hat{\beta}_{i}-\beta_{i}^*}_{\ell_1}
\leq 
4\lambda_T\left[\enVert[1]{\hat{\beta}_{i,J_{i}}-\beta^*_{i,J_{i}}}_{\ell_1}\wedge\enVert[1]{\beta^*_{i,J_{i}}}_{\ell_1}\right]\label{IQ2}\\
&\enVert[1]{\hat{\beta}_{i,{J^c_i}}-\beta^*_{i,{J^c_i}}}_{\ell_1}\leq 3\enVert[1]{\hat{\beta}_{i,J_{i}}-\beta^*_{i,J_{i}}}_{\ell_1}\label{IQ3}
\end{align}
\end{theorem}

The lower bound on the probability with which inequalities (\ref{IQ1})-(\ref{IQ3}) hold can be increased by choosing a larger value of $\lambda_T$. However, we shall see in Theorem \ref{Thm3} below that smaller values of $\lambda_T$ yield faster rates of convergence\footnote{In general, there is a tradeoff between $\lambda_T$ being small and the lower bound on the probability with which inequalities (\ref{IQ1})-(\ref{IQ3}) hold being large.}. 

It is important to notice that (\ref{IQ1})-(\ref{IQ3}) hold for all equations $i=1,...,k$ on one and the same set. This will turn out to be useful when deriving bounds on the estimation error for the whole system. Theorem \ref{Thm1} holds \textit{without any assumptions} on the Gram matrix. Furthermore, the lower bound on the probability with which inequalities (\ref{IQ1})-(\ref{IQ3}) hold is \textit{nonasymptotic} -- it holds for every $T$ -- and the above inequalities hold for \textit{any} configuration of $k,p,s_{i}$ and $T$. Note that the lower bound on the probability with which the estimates hold tends to one as $T\rightarrow\infty$.
In the course of the proof of Theorem \ref{Thm1} we derive a maximal inequality, Lemma \ref{Bound} in the appendix, for vector autoregressions which might be of independent interest. This inequality is rather sharp in the sense that it can be used to derive a rate of convergence of $\hat{\beta}_i,\ i=1,...,k$ which is within a logarithmic factor of the optimal $\sqrt{T}$ convergence rate. 


Inequalities (\ref{IQ1}) and (\ref{IQ2}) give immediate upper bounds on the prediction error, $\frac{1}{T}\enVert[1]{X\hat{\beta}_{i}-X\beta_{i}^*}^2$, as well as the estimation accuracy, $\enVert[0]{\hat{\beta}_{i}-\beta_{i}^*}_{\ell_1}$ of the LASSO. In particular, we shall use (\ref{IQ2}) to derive oracle inequalities for these two quantities in Theorem \ref{Thm3} below. Equation (\ref{IQ3}) is also of interest in its own right since it shows that an upper bound on the estimation error of the non-zero parameters in an equation will result in an upper bound on the estimation error of the zero parameters of that equation. This is remarkable since there may be many more zero parameters than non-zero ones in a sparsity scenario and since the bound does not depend on the relative size of the two groups of parameters.

The main contribution of theorem \ref{Thm1} is proving a lower bound on the probability of the event $\mathcal{B}_T=\cbr[0]{\max_{1\leq i\leq k}\enVert[0]{\frac{1}{T}X'\epsilon_i}_{\ell_\infty}\leq \lambda_T/2}$ for the choice of $\lambda_T$ made in Theorem \ref{Thm1}. In fact $\mathcal{B}_T$ is exactly the set referred to in this theorem and the lower bound on $P(\mathcal{B}_T)$ is $1-2(k^2p)^{1-\ln(1+T)}-2(1+T)^{-1/A}$ which is a lower bound on the probability with which the inequalities in Theorem \ref{Thm1} valid. Inequalities of the type (\ref{IQ1})-(\ref{IQ3}) are not new, see e.g. \cite{bickelrt09} or \cite{rigollett11}. What is novel is that they can be valid with high probability in a time series setting, i.e. the set $\mathcal{B}_T$ can have a high probability even in the presence of dependence.



\subsection{Restricted eigenvalue condition}
Theorem \ref{Thm1} did not pose any conditions on the (scaled) Gram matrix $\Psi_T$.  
If $kp>T$ the Gram matrix $\Psi_T$ is singular, or equivalently,
\begin{align}
\min_{\delta\in \mathbb{R}^{kp}\setminus\left\{0\right\}} \frac{\delta'\Psi_T\delta}{\enVert{\delta}^2}
=0.\label{singular}
\end{align}
In that case ordinary least squares is infeasible. However, for the LASSO \cite{bickelrt09} observed that the minimum in (\ref{singular}) can be replaced by a minimum over a much smaller set. The same is the case for the LASSO in the VAR since we have written the VAR as a regression model. In particular, we shall make use of the following restricted eigenvalue condition.
 
\begin{definition}[Restricted eigenvalue condition]
The restricted eigenvalue condition RE($r$) is said to be satisfied for some $1\leq r\leq kp$ if  
\begin{align}
\kappa_{\Psi_T}^2(r)=\min_{\substack{R\subseteq\cbr[0]{1,...,kp}\\ |R|\leq r}}\min_{\substack{ \delta\in \mathbb{R}^{kp}\setminus\left\{0\right\}\\ \enVert{\delta_{{R}^c}}_{\ell_1}\leq 3\enVert{\delta_{R}}_{\ell_1}}}\frac{\delta'\Psi_T\delta}{\enVert{\delta_{R}}^2}>0\label{REs}
\end{align}
where $R\subseteq \left\{1,...,kp\right\}$ and $|R|$ is its cardinality. 
\end{definition}
Instead of minimizing over all of $\mathbb{R}^{kp}$, the minimum in (\ref{REs}) is restricted to those vectors which satisfy $\enVert{\delta_{{R}^c}}_{\ell_1}\leq 3\enVert{\delta_{R}}_{\ell_1}$ and where $R$ has cardinality at most $r$. This implies that $\kappa_{\Psi_T}^2(r)$ in (\ref{REs}) can be larger than the minimized Rayleigh-Ritz ratio in (\ref{singular}) even when the latter is zero.

Notice that the restricted eigenvalue condition is trivially satisfied if $\Psi_T$ has full rank since $\delta_{R}'\delta_{R}\leq \delta'\delta$ for every $\delta\in \mathbb{R}^{kp}$ and so,
\begin{align*}
\frac{\delta'\Psi_T\delta}{\enVert{\delta_{R}}^2}
\geq 
\frac{\delta'\Psi_T\delta}{\enVert{\delta}^2}
\geq
\min_{\delta\in \mathbb{R}^{kp}\setminus\left\{0\right\}} \frac{\delta'\Psi_T\delta}{\enVert{\delta}^2}>0.
\end{align*}
This means that in the traditional setting of fewer variables per equation than observations the restricted eigenvalue condition is satisfied if $X'X$ is nonsingular. Hence, the results are applicable in this setting but also in many others. 
We shall be using the restricted eigenvalue condition with $r=s_{i}$ and denote $\kappa_{\Psi_T}^2(s_i)$ by $\kappa_{\Psi_{T,i}}^2$.

Let $\Gamma=E(\Psi_T)=E(Z_tZ_t')$ be the population covariance matrix of the data. We will assume that the corresponding restricted eigenvalue (defined like $\kappa_{\Psi_T}^2(s_i)$ in (\ref{REs})) $\kappa_{i}=\kappa_{\Gamma}(s_i)$ is strictly positive for all $i=1,...,k$. Note that this is satisfied in particular under the standard assumption that $\Gamma$ has full rank\footnote{Note that the full rank of the population covariance matrix $\Gamma$ is independent of the fact that one might have more variables than observations -- a fact which implies $\phi_{\min}(\Psi_T)=0$.}. To get useful bounds on the estimation error of the LASSO it turns out to be important that $\kappa_{\Psi_{T,i}}^2$ is not too small (in particular strictly positive). We show that this is the case (Lemma \ref{vdGB} in the appendix) as long as the random matrix $\Psi_T$ is sufficiently close to its expectation $\Gamma$. Hence, verifying the restricted eigenvalue condition is a question of showing that $\Psi_T$ is close to $\Gamma$ with high probability. To this end, Lemma \ref{CovBound} in the appendix gives finite sample bounds on the maximum entrywise distance between $\Psi_T$ and $\Gamma$ that hold with high probability. This result might be of independent interest in the theory of high-dimensional covariance estimation for dependent gaussian processes.
Lemma \ref{REbound} in the appendix uses this result to show that for any $0<q<1$ one has $P(\kappa_{\Psi_T,i}^2>q\kappa^2_{i})\geq 1-\pi_q(s_i)$ where $\pi_q(s_i)=4k^2p^2\exp\del[2]{\frac{-\zeta T}{s_i^2\log(T)(\log(k^2p^2)+1)}}+2(k^2p^2)^{1-\log(T)}$ for $\zeta=\frac{(1-q)^2\kappa_{i}^4}{4\cdot 16^3(\enVert[0]{\Gamma}\sum_{i=0}^T\enVert[0]{F^{i}})^2}$. The exponential decay of the first term of $\pi_q(s)$ hints at the fact that the restricted eigenvalue condition can be valid asymptotically even in high-dimensional systems. This will be explored in more detail in the next subsection. On the other hand, the probability of $\kappa_{\Psi_T,i}^2>q\kappa^2_{i}$ might be low in finite samples if, for example, $k$ or $p$ are very large. However, this underscores the conventional wisdom that one has to be careful with putting too much emphasis on the asymptotic results since these can be very misleading in finite samples.

The LASSO satisfies the following oracle inequalities in VAR models.

\begin{theorem}\label{Thm3}
Let $\lambda_T$ be as in Theorem \ref{Thm1} and $0<q<1$. Then with probability at least $1-2(k^2p)^{1-\ln(1+T)}-2(1+T)^{-1/A}-\pi_q(s_i)$ the following inequalities hold for all $i=1,...,k$ for some positive constant $A$.
\begin{align}
\frac{1}{T}\enVert[1]{X\hat{\beta}_{i}-X\beta_{i}^*}^2
&\leq 
\frac{16}{q\kappa_{i}^2}s_{i}\lambda_T^2\label{REIQ1}\\
\enVert[1]{\hat{\beta}_{i}-\beta_{i}^*}_{\ell_1}
&\leq
\frac{16}{q\kappa_{i}^2}s_{i}\lambda_T\label{REIQ2}
\end{align}
Furthermore, with at least the same probability as above, no relevant variables will be excluded from equation $i$ if $\beta_{\min,i}>\frac{16}{q\kappa_{i}^2}s_{i}\lambda_T$. Finally, all the above statements hold on one and the same set which has probability at least $1-2(k^2p)^{1-\ln(1+T)}-2(1+T)^{-1/A}-\pi_q(\bar{s})$. 
\end{theorem}
Notice that as in Theorem \ref{Thm1} the bounds are non-asymptotic. Inequality (\ref{REIQ1}) gives an upper bound on the prediction error compared to the hypothetical situation with knowledge of the true parameter vector. The more the restricted eigenvalue $\kappa_{i}$ is bounded away from zero, the smaller the upper bound on the prediction error. On the other hand, the prediction error is increasing in the number of non-zero parameters $s_{i}$. Finally, the prediction error is increasing in $\lambda_T$ but recall that $\lambda_T=\sqrt{8\ln(1+T)^5\ln(1+k)^4\ln(1+p)^2\ln(k^2p)\sigma_T^4/T}$ which implies $\lambda_T$ will be small for $\sigma_T$, $k$ and $p$ small and $T$ large. In the classical setting where $k,\ p$ and $s_i$ are fixed while $\kappa_i$ and $\sigma_T$ are bounded from below and above, respectively it is seen that the upper bound in (\ref{REIQ1}) is of order $O(\log(T)^5/T)$ which tends to zero almost as fast as $1/T$. A more detailed discussion of the role of $\sigma_T$, $k$ and $p$ can be found in the discussion following Corollary \ref{LASSOAsym} below.

Note that inequalities of the type in Theorem \ref{Thm3} are similar in spirit to the ones in Theorem 7.2 in \cite{bickelrt09}. Our main contribution here is to show that the restricted eigenvalue condition can be satisfied with high probability (and, as mentioned in connection with Theorem \ref{Thm1}, that $\mathcal{B}_T$ has high probability as well) in a time series context. This is not entirely trivial since $\kappa^2_{\Psi_T,i}$ is a rather involved function of many dependent variables. 

Inequality (\ref{REIQ2}) gives an upper bound on the estimation error of the LASSO. In the classical setting of a fixed model discussed above the right hand side is of order $O\del[1]{[\log(T)^{5}/T]^{1/2}}$ which tends to zero at a rate of almost $1/\sqrt{T}$. To shed further light on (\ref{REIQ2}) Lemma \ref{OLSO} below gives a corresponding result for the least squares estimator \textit{only including the relevant variables} -- i.e. least squares after the true sparsity pattern has been revealed by an oracle. To this end let $\beta_{OLS,i}$ denote the least squares estimator of $\beta_{i}^*$ only including the relevant variables.

\begin{lemma}\label{OLSO}
Let $\tilde{\lambda}_{T,i}=\sqrt{8\ln(1+T)^5\ln(1+s_i)^2\ln(s_i)\sigma_T^4/T}$ and $0<q<1$. If the true sparsity pattern is known and only the relevant variables are included in the model with their coefficients estimated by least squares equation by equation,
\begin{align}
\enVert[1]{\hat{\beta}_{OLS,i}-\beta^*_{i,J_{i}}}_{\ell_1}\leq \frac{\tilde{\lambda}_{T,i}}{2q\phi_{\min}(\Gamma_{J_{i},J_{i}})}s_{i}\label{OLSIN}
\end{align}
for all $i=1,...,k$ on a set with probability at least $1-2s_i^{1-\ln(1+T)}-2(1+T)^{-1/A}-\pi_q(s_i)$. 
\end{lemma}

Comparing (\ref{REIQ2}) to (\ref{OLSIN}) one notices that the upper bounds are very similar. Both expressions consist of $s_{i}$ multiplied by some term. Clearly this term is smaller for oracle assisted least squares, $\frac{\tilde{\lambda}_{T,i}}{2q\phi_{\min}(\Gamma_{J_{i},J_{i}})}$, than for the LASSO, $\frac{16\lambda_T}{q\kappa_{i}^2}$, since $\tilde{\lambda}_{T,i}\leq\lambda_T$ and $\kappa_{i}^2\leq \phi_{\min}(\Psi_{J_{i},J_{i}})$. However, $\lambda_T$ need not be much larger than $\tilde{\lambda}_{T,i}$ even if $kp$ is a lot larger than $s_i$ since the logarithmic function increases very slowly. Hence, it is reasonable to call (\ref{REIQ2}) an oracle inequality since it shows, in a non-asymptotic manner, that the LASSO performs almost as well as if one had known the true sparsity pattern and estimated the non-zero parameters by least squares. 

Also notice that the upper bounds on the $\ell_1$-estimation error in (\ref{REIQ2}) trivially yield upper bounds on the $\ell_p$-estimation error for any $p\geq 1$ since $\enVert[0]{\cdot}_{\ell_p}\leq \enVert[0]{\cdot}_{\ell_1}$ for any $1\leq p\leq \infty$. This observation is equally valid for all $\ell_1$-bounds in the sequel.


The last statement of Theorem \ref{Thm3} says that under the "beta-min" condition $\beta_{\min,i}>\frac{16}{q\kappa_{i}^2}s_{i}\lambda_T$ no relevant variables will be left out of the model. It is sensible that the beta-min condition is needed in order to be able to distinguish zero from non-zero parameters since the condition basically requires the two groups to be sufficiently separated -- the non-zero coefficients cannot be too close to zero. In particular, they must be bounded away from zero by a little more than the upper bound on the $\ell_1$ estimation error of the LASSO estimator. 

The following corollary to Theorem \ref{Thm3} gives upper bounds on the performance of the LASSO for the whole system.
\begin{corollary}\label{Corthm3}
Let $\lambda_T$ be as in Theorem \ref{Thm1} and $0<q<1$. Then, for some positive constant $A$
\begin{align}
\sum_{i=1}^k\enVert[1]{\hat{\beta}_{i}-\beta_{i}^*}_{\ell_1}
&\leq
\sum_{i=1}^{k}\frac{16}{q\kappa_i^2}s_{i}\lambda_T\label{REIQcor2}
\end{align}
with probability at least $1-2(k^2p)^{1-\ln(1+T)}-2(1+T)^{-1/A}-\pi_q(\bar{s})$ .
\end{corollary}
Corollary \ref{Corthm3} only gives an upper bound on the estimation error for the whole system since the systemwise counterparts of the other bounds in Theorem \ref{Thm3} are obtained in the same way. As can be expected the upper bound on the estimation error of the whole system is increasing in the number of variables. In the next section we shall investigate exactly how fast the size of the model can grow if one still wants systemwise consistency. 



\textbf{Remark 1:} Even though $kp$ can be a lot larger than $T$ the parameter $\beta_{i}^*$ in Theorem \ref{Thm3} is still uniquely defined since RE($s_{i}$) is assumed valid. This follows from an observation similar to observation 2 page 1721 in \cite{bickelrt09}.   

\textbf{Remark 2:} The above bounds also yield corresponding results for univariate autoregressions, i.e. for $k=1$. These follow trivially by setting $k=1$ in all the above bounds. 

\textbf{Remark 3:}\footnote{We wish to thank an anonymous referee for suggesting this remark.} The above models do not contain any exogenous variables. However, we conjecture that the results generalize to models with exogenous variables as long as these are stationary, sub-gaussian\footnote{A random variable $Z$ is said to be sub-gaussian if there exist positive constants $C$ and $K$ such that $P(|Z|\geq z)\leq K\exp(-Cz^2)$ for all $z>0$.} and are independent of the error terms since then these variables have the properties of the $y_t$ that are essential for the above analysis. To the extent that factors satisfy these assumptions these can also be included.

\subsection{Asymptotic properties of the Lasso}
All preceding results are for finite samples. In this section we utilize these results to describe the asymptotic properties of the LASSO as $T\rightarrow\infty$ to get a feeling for what size of models can be handled by our theory. 

\begin{theorem}\label{LASSOAsym}
Let $k,p\in O(e^{T^a})$ and $s_i\in O(T^b)$ for some $a,b\geq 0$. Assume that $7a+2b<1$ and that there exists a constant $c>0$ such that $\kappa^2_{i}\geq c$. Then, if $\sup_T\sigma_T,\ \sup_T\enVert[0]{\Gamma}\sum_{i=0}^T\enVert[0]{F^{i}}<\infty$, one has for $i=1,...,k$ as $T\rightarrow\infty$

\begin{itemize}
\item[i)]  $\frac{1}{T}\enVert[1]{X\hat{\beta}_{i}-X\beta_{i}^*}^2\rightarrow 0$ in probability
\item[ii)] $\enVert[0]{\hat{\beta}_{i}-\beta_{i}^*}_{\ell_1}\rightarrow 0$ in probability
\item[iii)] With probability tending to one no relevant variables will be excluded from the model if there exists a $T_0\geq 1$ such that $\beta_{\min,i}>\frac{16}{qc^2}s_{i}\lambda_T$ for all $T\geq T_0$. 
\end{itemize} 
\end{theorem}

Theorem \ref{LASSOAsym} shows that the parameters of each equation can be estimated consistently even when the number of variables is very large.
If one is only interested in the average prediction error tending to zero in probability, it suffices that $7a+b<1$. In either case $p$ and $k$ can increase at a sub-exponential rate -- and at the same time the number of relevant variables can arrive at a polynomial rate. The setting where the total number of parameters increases sub-exponentially in the sample size is often referred to as ultra-high or non-polynomial dimensionality. By choosing $a$ sufficiently close to 0, it is clear that any $b<1/2$ can be accommodated (the number of non-zero coefficients per equation cannot increase faster than the square root of the sample size) while still having the estimation error tending to zero in probability.

It may seem unreasonable to assume that the number of lags increases at an exponential rate in the sample size since this leads to having (asymptotically) more initial observations than there are observations in the sample. In the perhaps more realistic setting where $k$ increases at a sub-exponential rate while $p$ stays fixed, it can be shown that it suffices that $5a+2b<1$ to estimate $\beta^*_i$ consistently. When the sample is split into $p$ initial observations and $T-p$ observations for estimation the proof of Theorem \ref{LASSOAsym} reveals that for $p\in O(T^d)$ for some $0<d<1$ Theorem \ref{LASSOAsym} remains valid if $5a+2b<1$. It actually makes no difference whether we assume $p$ fixed or $p$ to increase at a rate slower than the sample size\footnote{Of course $p$ can not increase faster than the sample size in the setting where we set aside $p$ observations as initial observations and have $T-p$ observations left for estimation.}. The intuition behind this result is that $T-p$ is of the same order as $T$ in this setting so asymptotically it makes no difference whether $T-p$ or $T$ observations are available. Furthermore, it relies on the observation that $p$ only enters as a logarithm in the important places.


iii) gives a sufficient condition for avoiding excluding relevant variables asymptotically. 
The "beta-min" condition is necessary in the sense that one can not expect to be able to distinguish the non-zero coefficients from the zero ones if these are too close to each other. However, it should be noted that such an assumption rules out that the non-zero coefficients tend to zero too fast. On the other hand, it should also be said that the "beta-min" condition is not used in parts i) and ii) of Theorem \ref{LASSOAsym} -- it is only used to ensure that no relevant variables are excluded.

At this stage it is worth mentioning that the conditions in Theorem \ref{LASSOAsym} are merely sufficient. For example one can loosen the assumption of the boundedness of the suprema $\sup_T\sigma_T$ and $\sup_T\enVert[0]{\Gamma}\sum_{i=0}^T\enVert[0]{F^{i}}$ by tightening $7a+2b<1$. We have only discussed the cases of $p$ increasing exponentially fast or $p$ constant above. Other settings in between those two extremes are of course also relevant. Also notice that the beta-min condition in iv) is satisfied in particular if there exists a constant $\hat{c}>0$ such that $\beta_{\min,i}\geq \hat{c}$ since $s_i\lambda_T\rightarrow 0$ by (\ref{slambdaorder}) in the appendix. 

Furthermore, it is of interest to investigate the behavior of the LASSO estimator as the largest root of the companion matrix $F$, denoted $\rho$, tends to the boundary of the unit disc. The next theorem deals with the estimation error in a non-triangular array model (the model stays constant over time $T$).

\begin{theorem}\label{Lassoloc}
Let $k,p$ and $s$ be fixed and assume that there exists a constant $c>0$ such that $\kappa_i^2>c$ and that $F$ has distinct eigenvalues. Then, letting $\rho=1-\frac{1}{T^\alpha}$ with $\alpha<1/4$, one has for all $i=1,...,k$ 
\begin{align*}
\enVert[0]{\hat{\beta}_{i}-\beta_{i}^*}_{\ell_1}\rightarrow 0 \text{ in probability}
\end{align*}
\end{theorem}
Theorem \ref{Lassoloc} shows that as long as the rate with which $\rho\to 1$ is slower than $T^{-1/4}$, consistent estimation of $\beta^*_i$ is possible for all $1\leq i\leq  k$ even in the local to unity setting. $\rho$ can not tend to one too fast for the following reasons. Firstly, $\max_{1\leq i\leq k}\sigma_{y,i}$ increases without bound as $\rho$ tends to the boundary of the unit disc. This increases $\sigma_T$ and hence $\lambda_T$ in (\ref{REIQ2}). Secondly, and as can be seen from the proof of the theorem, the probability with which (\ref{REIQ2}) is valid is also decreasing in $\rho$.  The case where the model is allowed to change with the sample size ($k,p$ and $s$ not fixed) is much more subtle since the proof of Theorem \ref{Lassoloc} uses the diagonalization $F=V^{-1}DV$, where $V$ is a matrix whose columns are the linearly independent eigenvectors of $F$ and $D$ is a diagonal matrix containing the eigenvalues of $F$, to conclude that for a \textit{fixed} $F$ there exists a $C_F>1$ such $\enVert[0]{F^j}\leq C_F\rho^j$.  In general $C_F$ depends on $F$ and hence the case where the model, and hence $F$, is allowed to vary with the sample size is not covered by our techniques. Of course, the fixed $F$ case still gives a useful upper bound on the speed with which $\rho$ can tend to one since it can not be faster in the general case of varying $F$.\footnote{The assumption of $F$ having distinct eigenvalues is sufficient, though not necessary, for the diagonalization $F=V^{-1}DV$ to exist. This diagonalization is convenient since it yields $\enVert[0]{F^j}\leq C_F\rho^j$. Alternatively, a Jordan decomposition could be applied but this would not give a diagonal $D$ and hence $\enVert[0]{F^j}$ can not as easily be bounded in terms of $\rho^j$.}

Regarding systemwise consistency, we have the following result which is a consequence of Corollary \ref{Corthm3}.

\begin{theorem}\label{systemlassoasym}
Let $p\in O(e^{T^a})$, $k\in O(T^b)$ and $\bar{s}\in O(T^c)$ for some $a,b,c\geq 0$. Assume that $3a+2b+2c<1$ and that there exists a constant $d>0$ such that $\kappa_{i}\geq d$. If $\sup_T\sigma_T,\ \sup_T\enVert[0]{\Gamma}\sum_{i=0}^T\enVert[0]{F^{i}}<\infty$ one has that as $T\rightarrow \infty$
$\sum_{i=1}^k\enVert[0]{\hat{\beta}_{i}-\beta_{i}^*}_{\ell_1}\rightarrow 0$ in probability.
\end{theorem}
Theorem \ref{systemlassoasym} reveals that the requirements for systemwise consistency are a lot stricter than the equationwise ones. In particular, $k$ can now only increase at a polynomial rate as opposed to the sub-exponential rate in Theorem \ref{LASSOAsym}. However, it is sensible that the number of equations cannot increase too fast if one wishes the sum of the estimation errors to vanish asymptotically. 

If $k,p$ and $s$ are fixed numbers (not depending on $T$) as in the classical setting then $\sum_{i=1}^k\enVert[0]{\hat{\beta}_{i}-\beta_{i}^*}_{\ell_1}\in O_p(\lambda_T)=O_p(\sqrt{\ln(1+T)^5/T})$ revealing that the rate of convergence is almost $\sqrt{T}$. While the logarithmic factor can be lowered (see also footnote \ref{smalllambda}) we don't think it is possible to remove it altogether (using the techniques in this paper). In the case where $p$ is fixed, corresponding to $a=0$, one only needs $2b+2c<1$ in order to obtain systemwise consistency.

Bounds on the systemwise prediction error and total number of variables selected can be obtained in a similar fashion as the bounds on the estimation error for the whole system. Again the case $k=1$ gives results corresponding to univariate autoregressions.

At this point we should also mention the possibility of using the Group Lasso of \cite{yuanl06}\footnote{We would like to thank an anonymous referee for pointing this out.}. The idea here is that variables can be grouped and that entire groups of variables could be irrelevant and hence the concept of group sparsity arises. Natural groups in the VAR in (\ref{VAR1}) could be the parameter matrices $\cbr[0]{\Phi_l}_{l=1}^p$ corresponding to a certain lag length being irrelevant. Or, alternatively the groups could be the parameters of lags of a certain variable corresponding to this variable being irrelevant in case all coefficients are zero. The adaptive group Lasso for VAR models has been studied in an asymptotic framework in \cite{kockc12oracle}.

\section{The adaptive Lasso}\label{SecAdaLASSO}
The LASSO penalizes all parameters equally. If it were possible to penalize the truly zero parameters more than the non-zero ones, one would expect a better performance. \cite{zou06} used this idea to propose the adaptive LASSO in the standard linear regression model with a fixed number of non-random regressors. He established that the adaptive LASSO is asymptotically oracle efficient in this setting -- with probability tending to one it selects the correct sparsity pattern. We now apply the adaptive LASSO to our vector autoregressive model. We shall give lower bounds on the finite sample probabilities of selecting the correct model. Then, these bounds are used to establish that with probability tending to one the correct sparsity pattern (and a little bit more) is unveiled. 

The adaptive LASSO estimates $\beta_{i}^*$ by minimizing the following objective function
\begin{align}
\tilde{L}(\beta_{i})=\frac{1}{T}\enVert{y_{i}-X_{J(\hat{\beta}_{i})}\beta_{i,J(\hat{\beta}_{i})}}^2+2\lambda_T\sum_{j\in J(\hat{\beta}_{i})}\frac{|\beta_{i,j}|}{|\hat{\beta}_{i,j}|},\ i=1,...,k\label{ALASSOobj}
\end{align}
where $\hat{\beta}_{i,j}$ denotes the LASSO estimator of $\beta^*_{i,j}$ from the previous section. Let $\tilde{\beta}_{i}$ denote the minimizer of (\ref{ALASSOobj}). Note that if $\hat{\beta}_{i,j}=0$ the $j$'th variable is excluded from the $i$th equation. If the first stage LASSO estimator classifies a parameter as zero it is not included in the second step resulting in a problem of a much smaller size. If $\beta_{i,j}^*=0$ then $\hat{\beta}_{i,j}$ is likely to be small by (\ref{REIQ2}) and consistency of the LASSO. Hence, $1/\hat{\beta}_{i,j}$ is large and the penalty on $\beta_{i,j}$ is large. If  $\beta_{i,j}^*\neq 0$, $\hat{\beta}_{i,j}$ is not too close to zero and the penalty is small. In short, the adaptive LASSO is a two step estimator with which greater penalties are applied to the truly zero parameters. These more intelligent weights allow us to derive conditions under which the adaptive LASSO is sign consistent, i.e. it selects the correct sign pattern. This in particular implies that the correct sparsity pattern is chosen.

Even though we use the LASSO as our initial estimator, this is not necessary. All we shall make use of is the upper bound on its $\ell_1$-estimation error. Hence, the results in Theorem \ref{AdaLasso} below can be improved if an estimator with tighter bounds is used. Furthermore, we have chosen the same value of $\lambda_T$ in (\ref{ALASSOobj}) as in (\ref{LASSOobj}). This is not necessary but it turns out that this is a convenient choice since in particular (\ref{AdaLasso1}) below becomes simpler. 

The first Theorem gives lower bounds on the \textit{finite sample probability} of the adaptive LASSO being sign-consistent. 

\begin{theorem}\label{AdaLasso}
Let $\lambda_T$ be as above and assume that\footnote{It suffices that $\beta_{\min,i}>\enVert[0]{\hat{\beta}_{i}-\beta_{i}^*}_{\ell_1}$ such that $\beta_{\min,i}\geq q\enVert[0]{\hat{\beta}_{i}-\beta_{i}^*}_{\ell_1}$ for some $q>1$.} $\beta_{\min,i}\geq 2\enVert[0]{\hat{\beta}_{i}-\beta_{i}^*}_{\ell_1}$ and
\begin{align}
&\frac{s_{i}K_T}{q\phi_{\min}(\Gamma_{J_{i},J_{i}})}\left(\frac{1}{2}+\frac{2}{\beta_{\min,i}}\right)\enVert[0]{\hat{\beta}_{i}-\beta_{i}^*}_{\ell_1}+\frac{\enVert[0]{\hat{\beta}_{i}-\beta_{i}^*}_{\ell_1}}{2}
\leq 
1 \label{AdaLasso1}\\
&\frac{\sqrt{s_{i}}}{q\phi_{\min}(\Gamma_{J_{i},J_{i}})}\left(\frac{\lambda_T}{2}+\frac{2\lambda_T}{\beta_{\min,i}}\right)
\leq
\beta_{\min,i}\label{AdaLasso2}
\end{align}
where $K_T=\ln(1+k)^2\ln(1+p)^2\ln(T)\sigma_T^2$. Then, on a set with probability at least $1-2(k^2p)^{1-\ln(1+T)}-2(1+T)^{-1/A}-2T^{-1/A}-\pi_q(s_i)$ it holds that $\sgn(\tilde{\beta}_{i})=\sgn(\beta_{i}^*)$ for all $i=1,...,k$.
\end{theorem}
Here we have chosen to keep the expressions at a high level instead of inserting the upper bound on $\enVert[0]{\hat{\beta}_{i}-\beta_{i}^*}_{\ell_1}$ from Theorem \ref{Thm3} since this facilitates the interpretation. We also note, that the probabilities of detecting the correct sign pattern may be very small in small samples but the above result will be very useful in establishing the asymptotic sign consistency below.

As in Theorem \ref{Thm3}, $\sgn(\tilde{\beta}_{i})=\sgn(\beta_{i}^*)$ can be constructed to hold on the same set for all $i=1,...,k$ by choosing $s_i=\bar{s}$ in Theorem \ref{AdaLasso}. Clearly, the more precise the initial estimator, the smaller the left hand side in (\ref{AdaLasso1}). On the other hand a small $\beta_{\min,i}$ makes the inequality harder to satisfy. This is sensible since the correct sign pattern is harder to detect if the non-zero parameters are close to zero. $K_T$ is increasing in the dimension of the model and so large $k$ and $p$ make it harder to detect the correct sign pattern. The interpretation of (\ref{AdaLasso2}) is similar since $\lambda_T$ is increasing in the dimension of the model. Notice again that the assumption $\beta_{\min,i}\geq 2\enVert[0]{\hat{\beta}_{i}-\beta_{i}^*}_{\ell_1}$ is a reasonable one: one cannot expect to detect the correct sign pattern if the precision of the initial estimator is smaller than the distance the smallest non-zero coefficient is bounded away from zero by since otherwise the initial LASSO estimator may falsely classify non-zero parameters as zero.

Also notice that by the last assertion of Theorem \ref{Thm3}, $\beta_{\min,i}\geq 2\enVert[0]{\hat{\beta}_{i}-\beta_{i}^*}_{\ell_1}$ ensures that the initial LASSO estimator will not exclude any relevant variables. This is of course a necessary condition for the second stage adaptive LASSO to select the correct sign pattern.  

\subsection{Asymptotic properties of the adaptive Lasso}
The results in Theorem \ref{AdaLasso} are non-asymptotic but can be used to obtain the following sufficient conditions for asymptotic sign consistency of the adaptive LASSO.  

\begin{theorem}\label{AdaLASSOAsym} 
Assume that there exists a $\tilde{c}_{i}>0$ such that $\kappa_{i}\geq \tilde{c}_{i}$ and that $\sup_{T}\sigma_T\leq \infty$ as well as $\sup_T\enVert[0]{\Gamma}\sum_{i=0}^T\enVert[0]{F^{i}}<\infty$. If furthermore $k,p\in O(e^{T^a})$ as well as $s_{i}\in O(T^b)$ for some $a,b\geq 0$ satisfying $15a+4b<1$ and $\beta_{\min,i}\in\Omega(\ln(T)[a_T\vee b_T])$ for $a_T=T^{2b}T^{(15/2)a-1/2}\ln(T)^{1+5/2}$ and $b_T= T^{b/4}T^{(7/4)a-1/4}\ln(T)^{5/4}$, then $P(\sgn(\tilde{\beta}_{i})=\sgn(\beta_{i}^*))\rightarrow 1$.\footnote{Here $f(T)\in\Omega(g(T))$ means that there exists a constant $c$ such that $f(T)\geq cg(T)$ from a certain $T_0$ and onwards. Thus there exists a constant $c$ such that $\beta_{\min}\geq c\ln(T)[a_T\vee b_T]$ from a $T_0$ and onwards.}
\end{theorem}
Note that the requirements on $a$ and $b$ are stronger than in Theorem \ref{LASSOAsym} but the number of included variables may still be much larger than the sample size. The number of relevant variables must now be $o(T^{1/4})$. How small can $\beta_{\min,i}$ be? To answer this, consider a model with fixed $k$ and $p$ corresponding to $a=b=0$. This implies that $\beta_{\min,i}\in \Omega(\ln(T)^{9/4}T^{-1/4})$. Of course this is the case in particular if there exists a $d>0$ such that $\beta_{\min,i}\geq d$. As mentioned previously, the non-zero coefficients cannot be too small if one wishes to recover the correct sparsity pattern. If the non-zero coefficients tend to zero too fast it is well known that consistent model selection techniques will classify them as zero, see e.g. \cite{leebp05} or \cite{kock12} for a time series context. The beta-min condition exactly guards against non-zero coefficients shrinking to zero too fast.

The above conditions are merely sufficient. For example it is possible to relax $\sup_{T}\sigma_T<\infty$ or $\sup_T\enVert[0]{\Gamma}\sum_{i=0}^T\enVert[0]{F^{i}}<\infty$ at the price of slower growth rates for $s_{i},k$ and $p$. In the perhaps more realistic setting where $p$ is fixed and only $k\in O(e^{T^a})$ the conditions of Theorem \ref{AdaLASSOAsym} can be relaxed to $9a+4b<1$ while one may choose $a_T=T^{2b}T^{(9/2)a-1/2}\ln(T)^{1+5/2}$ and $b_T=T^{b/4}T^{(5/4)a}T^{-1/4}\ln(T)^{5/4}$. Note however, that for $a=b=0$ we still require  $\beta_{\min,i}\in \Omega(\ln(T)^{9/4}T^{-1/4})$. Hence the speed at which the smallest non-zero coefficient tends to zero does not increase.

It is also worth mentioning that the conclusion of Theorem \ref{AdaLASSOAsym} can be strengthened to $P(\cap_{i=1}^k\cbr[0]{\sgn(\tilde{\beta}_{i})=\sgn(\beta_{i}^*)})\rightarrow 1$ if the conditions are made uniform in $i=1,...,k$\footnote{We don't state the full theorem here since it basically entails deleting subscript $i$ and replacing $s_i$ by $\bar{s}=\max\cbr[0]{s_1,...,s_k}$ and  $\phi_{\min}(\Psi_{J_{i},J_{i}})$ and $\kappa_{i}$ by the corresponding versions minimized over $i$.}.
 
Furthermore, we remark that Theorems \ref{AdaLasso} and \ref{AdaLASSOAsym} show that the adaptive LASSO can be used to investigate whether a certain variable Granger causes another or not since these two theorems give conditions under which the adaptive LASSO detects the correct sparsity pattern\footnote{Of course model selection mistakes may occur and hence the adaptive Lasso is no panacea.}.
 
Finally, we show that the estimates of the non-zero parameters of the adaptive LASSO are asymptotically equivalent to the least squares ones only including the relevant variables. Hence, the limiting distribution of the non-zero coefficients is identical to the oracle assisted least squares estimator.
\begin{theorem}\label{asymdist}
Let the assumptions of Theorem \ref{AdaLASSOAsym} be satisfied and $\alpha_i$ be an $s_i\times 1$ vector with unit norm. Then,
\begin{align*}
\envert[1]{\sqrt{T}\alpha_{i}'(\tilde{\beta}_{J_{i}}-\beta^*_{J_{i}})-\sqrt{T}\alpha_{i}'(\hat{\beta}_{OLS,i}-\beta^*_{J_i})}\in o_p(1)
\end{align*}
where $o_p(1)$ is a term that converges to zero in probability uniformly in $\alpha_i$.
\end{theorem} 
Theorem \ref{asymdist} reveals that $\sqrt{T}\alpha_{i}'(\tilde{\beta}_{J_{i}}-\beta^*_{J_i})$ is asymptotically equivalent to $\sqrt{T}\alpha_{i}'(\hat{\beta}_{OLS,i}-\beta^*_{J_{i}})$. Thus inference is asymptotically as efficient as oracle assisted least squares. As seen from the discussion following Theorem \ref{AdaLASSOAsym} this is the case in even very high-dimensional models. For the case of $p$ fixed one may adopt the milder assumptions discussed for that case after Theorem \ref{AdaLASSOAsym}. The adaptive LASSO remains asymptotically equivalent to the oracle assisted least squares procedure.  

By combining Theorem \ref{asymdist} and Lemma \ref{OLSO} one obtains the following upper bound on the rate of convergence of $\tilde{\beta}_{i}$ to $\beta_{i}^*$. 

\begin{corollary}\label{corrate}
Let the assumptions of part ii) of Theorem \ref{AdaLASSOAsym} be satisfied. Then,
\begin{align*}
\enVert[0]{\tilde{\beta}_{J_i}-\beta^*_{J_i}}_{\ell_1}\in O_p\left(\tilde{\lambda}_{T,i}s_{i}\right)
\end{align*}
where as in Lemma \ref{OLSO} $\tilde{\lambda}_{T,i}=\sqrt{8\ln(1+T)^5\ln(1+s_{i})^2\ln(s_{i})/T}$.
\end{corollary}
Notice that the rate of convergence is as fast as the one for the oracle assisted least squares estimator obtained by Lemma \ref{OLSO}. Hence, the adaptive LASSO improves further on the LASSO by selecting the correct sparsity pattern and estimating the non-zero coefficients at same rate as the least squares oracle. It is not difficult to show that in the case of fixed covariates the oracle assisted least squares estimator satisfies $\enVert[0]{\tilde{\beta}_{OLS,i}-\beta^*_{J_{i}}}\in O_p(s_{i}/\sqrt{T})$. Hence, we conjecture that it may be possible to decrease $\tilde{\lambda}_{T,i}$ in Corollary \ref{corrate} to $1/\sqrt{T}$ but in any case the current additional factors are merely logarithmic.

\section{Monte Carlo}\label{MC}
This section explores the finite sample properties of the LASSO and the adaptive LASSO. We compare the performance of these procedures to oracle assisted least squares which is least squares including only the relevant variables. This estimator is of course unfeasible in practice but is nevertheless a useful benchmark. Whenever the sample size permits it we also implement least squares including all variables, i.e. without any variable selection whatsoever. This is at the other extreme of Oracle OLS. We also implement (when feasible) the post Lasso estimator, that is, estimating least squares using the variables selected by the LASSO. Finally, the adaptive LASSO is also implemented using ridge regression as the first step estimator\footnote{We would like to thank two anonymous referees for suggesting using post-LASSO and ridge regression as initial estimator.}.  All the computations are implemented using the \texttt{lassovar} package (based on the \texttt{glmnet} algorithm) for {\tt R} and are fully replicable using the supplementary material provided. 

To select the value of the tuning parameter $\lambda_T$ we use BIC defined as $BIC_{\lambda_T}=\log(RSS)+\frac{\log(T)}{T}df\left(\lambda_T \right)$ where $RSS$ is the sum of squared residuals (dependence on $\lambda_T$ suppressed) and $df(\lambda_T)$ the degrees of freedom of the model for a fixed $\lambda_T$. Following \cite{bühlmannvdg11} we use $df\left( \lambda_T \right)=|\hat{J}\left( \lambda_T \right)|$ as an unbiased estimator of the degrees of freedom of the LASSO and $df\left( \lambda_T\right)=trace(X\left( X'X + \lambda_T I \right)^{-1}X')$ for the ridge regression. We also experimented with cross validation but this did not improve the results while being considerably slower. 

All procedures are implemented equation by equation and their performance is measured along the following dimensions which are reported for the whole system. 

\begin{enumerate}
\item Correct sparsity pattern: How often does a procedure select the correct sparsity pattern for all equations, i.e. how often does it include all the relevant variables while discarding all irrelevant variables.
\item True model included: How often does a procedure retain the relevant variables in all equations. This is a relevant measure in practice since even if a procedure does not detect the correct sparsity pattern it may still be able to retain all relevant variables while hopefully leaving many irrelevant variables out and hence reducing the dimension of the model.  
\item Fraction of relevant variables included. Even if a procedure wrongly discards a relevant variable it is still relevant to know how big a fraction of vari 
\item Number of variables included: How many variables does each procedure include on average. This measures how well a procedure reduces the dimension of the problem.
\item RMSE: The root mean square error of the parameter estimates calculated as $\sqrt{\frac{1}{MC}\sum_{i=1}^{MC}\enVert[1]{\hat{\beta}^{(i)}-\beta^*}^2}$ where $MC$ denotes the number of Monte Carlo replication and $\hat{\beta}^{(i)}$ is the estimated parameter vector in the $i$th Monte Carlo replication by any of the above mentioned procedures.
\item 1-step ahead RMSFE: For every Monte Carlo replication the estimated parameters are used to make a one step ahead forecast of the whole vector $y^{(i)}_{T+1}$ denoted $\hat{y}^{(i)}_{T+1,T}$. The root mean square forecast error (RMSFE) is calculated as 
$\sqrt{\frac{1}{k}\frac{1}{MC}\sum_{i=1}^{MC}\enVert[1]{\hat{y}^{(i)}_{T+1,T}-y^{(i)}_{T+1}}^2}$.
\end{enumerate} 
The following three Experiments are considered where the covariance matrix of the error terms is diagonal with $.01$ on the diagonal in all settings. The sample sizes are $T=50,100$ and $500$.

\begin{itemize}
\item Experiment A: The data is generated from a VAR(1) model with $\Phi_1=diag(0.5,...,0.5)$ and with $k=10, 20, 50$ and $100$. This is a truly sparse model where the behavior of each variable only depends on its own past. The case $k=100$ illustrates a high dimensional setting where each equation has 99 redundant variables.
\item Experiment B: The data is generated from a VAR(4) model where $\Phi_1$ and $\Phi_4$ have a block diagonal structure. In particular, the blocks are $5\times 5$ matrices with all entries of the blocks of $\Phi_1$ equal to $.15$ and all elements of the blocks of $\Phi_4$ equal to $-.1$. $\Phi_2=\Phi_3=0$. The largest root of the companion matrix of the system is $.98$ indicating a very persistent behavior of the system. This structure could be motivated by a model build on quarterly data as is often the case in macroeconometrics. $k=10,20$ and $50$.
\item Experiment C: The data is generated from a VAR(5) model where $\Phi_1=diag(.95,...,.95)$ and $\Phi_j=(-.95)^{(j-1)}\Phi_1,\ j=2,...,5$. This results in a system with a companion matrix that has a maximal eigenvalue of $.92$. The coefficients get smaller on distant lags reflecting the conventional wisdom that recent lags are more important than distant ones. $k=10, 20$ and $50$.
\item Experiment D: The data is generated from a VAR(1) with the $(i,j)$th entry given by $(-1)^{|i-j|}\rho^{|i-j|+1}$ with $\rho=0.4$. Hence, the entries decrease exponentially fast in the distance from the diagonal. This setting illustrates a violation of the sparsity assumption since no parameters are zero. Furthermore, there are many small but non-zero parameters.
\end{itemize}

	\begin{sidewaystable}
\resizebox{\linewidth}{!}{%
	\begin{tabular}{l rrr rrr rrr rrr rrr rrr}
	
\toprule
& \multicolumn{ 3 }{c}{ LASSO }& \multicolumn{ 3 }{c}{ post-LASSO }& \multicolumn{ 3 }{c}{ adaptive LASSO }& \multicolumn{ 3 }{c}{ adaptive LASSO }& \multicolumn{ 3 }{c}{ Oracle OLS }& \multicolumn{ 3 }{c}{Full  OLS }\\
&\multicolumn{ 6 }{c}{ } &\multicolumn{ 3 }{c}{ $1^{st}$ step: LASSO} &\multicolumn{ 3 }{c}{ $1^{st}$ step: Ridge } &\multicolumn{ 6 }{c}{  } \\
\cmidrule(r){ 2 - 4 }\cmidrule(r){ 5 - 7 }\cmidrule(r){ 8 - 10 }\cmidrule(r){ 11 - 13 }\cmidrule(r){ 14 - 16 }\cmidrule(r){ 17 - 19 }T &50 & 100 & 500 & 50 & 100 & 500 & 50 & 100 & 500 & 50 & 100 & 500 & 50 & 100 & 500 & 50 & 100 & 500\\
\midrule
k&\multicolumn{ 18 }{c}{ True model uncovered. }\\
\midrule
 10 & 0.00 & 0.01 & 0.10 & 0.00 & 0.01 & 0.10 & 0.00 & 0.03 & 0.35 & 0.00 & 0.02 & 0.49 & 1.00 & 1.00 & 1.00 & 0.00 & 0.00 & 0.00 \\ 
  20 & 0.00 & 0.00 & 0.01 & 0.00 & 0.00 & 0.01 & 0.00 & 0.00 & 0.04 & 0.00 & 0.00 & 0.12 & 1.00 & 1.00 & 1.00 & 0.00 & 0.00 & 0.00 \\ 
  50 & 0.00 & 0.00 & 0.00 &  & 0.00 & 0.00 & 0.00 & 0.00 & 0.00 & 0.00 & 0.00 & 0.00 & 1.00 & 1.00 & 1.00 &  & 0.00 & 0.00 \\ 
  100 & 0.00 & 0.00 & 0.00 &  & 0.00 & 0.00 & 0.00 & 0.00 & 0.00 & 0.00 & 0.00 & 0.00 & 1.00 & 1.00 & 1.00 &  &  & 0.00 \\ 
  \midrule
&\multicolumn{ 18 }{c}{ True model included. }\\
\midrule
 10 & 0.06 & 0.80 & 1.00 & 0.06 & 0.80 & 1.00 & 0.06 & 0.79 & 1.00 & 0.14 & 0.91 & 1.00 & 1.00 & 1.00 & 1.00 & 1.00 & 1.00 & 1.00 \\ 
  20 & 0.00 & 0.43 & 1.00 & 0.00 & 0.43 & 1.00 & 0.00 & 0.44 & 1.00 & 0.00 & 0.71 & 1.00 & 1.00 & 1.00 & 1.00 & 1.00 & 1.00 & 1.00 \\ 
  50 & 0.00 & 0.03 & 1.00 &  & 0.03 & 1.00 & 0.00 & 0.01 & 1.00 & 0.00 & 0.27 & 1.00 & 1.00 & 1.00 & 1.00 &  & 1.00 & 1.00 \\ 
  100 & 0.00 & 0.00 & 1.00 &  & 0.00 & 1.00 & 0.00 & 0.00 & 1.00 & 0.00 & 0.01 & 1.00 & 1.00 & 1.00 & 1.00 &  &  & 1.00 \\ 
  \midrule
&\multicolumn{ 18 }{c}{ Share of relevant variables selected. }\\
\midrule
 10 & 0.75 & 0.98 & 1.00 & 0.75 & 0.98 & 1.00 & 0.74 & 0.98 & 1.00 & 0.83 & 0.99 & 1.00 & 1.00 & 1.00 & 1.00 & 1.00 & 1.00 & 1.00 \\ 
  20 & 0.67 & 0.96 & 1.00 & 0.67 & 0.96 & 1.00 & 0.66 & 0.96 & 1.00 & 0.78 & 0.98 & 1.00 & 1.00 & 1.00 & 1.00 & 1.00 & 1.00 & 1.00 \\ 
  50 & 0.69 & 0.93 & 1.00 &  & 0.93 & 1.00 & 0.57 & 0.93 & 1.00 & 0.71 & 0.97 & 1.00 & 1.00 & 1.00 & 1.00 &  & 1.00 & 1.00 \\ 
  100 & 0.49 & 0.91 & 1.00 &  & 0.91 & 1.00 & 0.30 & 0.90 & 1.00 & 0.73 & 0.96 & 1.00 & 1.00 & 1.00 & 1.00 &  &  & 1.00 \\ 
  \midrule
&\multicolumn{ 18 }{c}{ Number of variables selected. }\\
\midrule
 10 &  15 &  15 &  12 &  15 &  15 &  12 &  13 &  13 &  11 &  15 &  13 &  10 &  10 &  10 &  10 & 100 & 100 & 100 \\ 
  20 &  37 &  34 &  25 &  37 &  34 &  25 &  32 &  31 &  23 &  42 &  32 &  22 &  20 &  20 &  20 & 400 & 400 & 400 \\ 
  50 & 930 &  92 &  66 &  &  92 &  66 & 635 &  88 &  65 & 212 & 104 &  59 &  50 &  50 &  50 &  & 2500 & 2500 \\ 
  100 & 4769 & 214 & 135 &  & 214 & 135 & 2994 & 199 & 134 & 2080 & 259 & 124 & 100 & 100 & 100 &  &  & 10000 \\ 
  \midrule
&\multicolumn{ 18 }{c}{ Root mean square estimation error. }\\
\midrule
 10 & 1.13 & 0.71 & 0.28 & 1.18 & 0.62 & 0.20 & 1.12 & 0.56 & 0.17 & 1.04 & 0.56 & 0.16 & 0.44 & 0.30 & 0.12 & 1.69 & 1.03 & 0.40 \\ 
  20 & 1.79 & 1.15 & 0.45 & 1.98 & 1.03 & 0.31 & 1.91 & 0.97 & 0.28 & 1.72 & 0.91 & 0.24 & 0.63 & 0.41 & 0.18 & 4.13 & 2.30 & 0.82 \\ 
  50 & 8.37 & 2.08 & 0.82 &  & 1.95 & 0.55 & 7.70 & 1.91 & 0.53 & 3.53 & 1.71 & 0.43 & 0.97 & 0.65 & 0.28 &  & 7.84 & 2.22 \\ 
  100 & 12.09 & 3.26 & 1.26 &  & 3.21 & 0.83 & 12.34 & 2.97 & 0.81 & 8.97 & 2.73 & 0.64 & 1.39 & 0.93 & 0.39 &  &  & 4.98 \\ 
  \midrule
&\multicolumn{ 18 }{c}{ Root mean square 1-step ahead forecast error. }\\
\midrule
 10 & 0.107 & 0.104 & 0.099 & 0.108 & 0.103 & 0.099 & 0.108 & 0.102 & 0.101 & 0.108 & 0.102 & 0.100 & 0.101 & 0.101 & 0.100 & 0.116 & 0.108 & 0.101 \\ 
  20 & 0.109 & 0.104 & 0.101 & 0.111 & 0.104 & 0.100 & 0.110 & 0.104 & 0.101 & 0.108 & 0.102 & 0.101 & 0.103 & 0.102 & 0.100 & 0.138 & 0.115 & 0.103 \\ 
  50 & 0.154 & 0.105 & 0.101 &  & 0.104 & 0.100 & 0.146 & 0.105 & 0.100 & 0.112 & 0.104 & 0.100 & 0.102 & 0.101 & 0.100 &  & 0.149 & 0.106 \\ 
  100 & 0.151 & 0.107 & 0.101 &  & 0.106 & 0.100 & 0.153 & 0.105 & 0.101 & 0.129 & 0.104 & 0.100 & 0.102 & 0.102 & 0.100 &  &  & 0.113 \\ 
  \bottomrule

	\end{tabular}
}
\caption{\small The results for Experiment A measured along the dimensions discussed in the main text.}
\label{Table1}
	\end{sidewaystable}

Table \ref{Table1} contains the results for Experiment A. Blank entries indicate settings where a procedure was not feasible. 

Neither the LASSO nor the adaptive LASSO unveil the correct sparsity pattern very often. However, in accordance with Theorem \ref{AdaLASSOAsym}, the adaptive LASSO shows a clear improvement along this dimension as the sample size increases when $k=10$. On the other hand, detecting exactly the correct model might be asking for too much. This is illustrated by the fact that the LASSO as well as the adaptive LASSO very often include all relevant variables. Table \ref{Table1} also shows that even in the cases where the true model is not included in the set chosen by the LASSO or the adaptive LASSO, the share of relevant variables included is still relatively high. The worst performance may be found for $k=100$ and $T=50$ where the share of relevant variables included by the adaptive LASSO is 30 percent. Also notice that when the LASSO is used as the initial estimator for the adaptive LASSO the latter can perform no better along this dimension than the former (variables excluded in the first step are also excluded in the second step). In this light it is encouraging that the adaptive LASSO actually performs almost as well as the LASSO -- it rarely discards any relevant variables in the second step. But how many variables are included in total, or put differently, how well do the procedures reduce the dimension of the model? For this measure the results are quite encouraging. Even when $k=100$ only 134 variables out $10,000$ possible are included by the adaptive LASSO when $T=500$. Since the relevant variables are always included this means that only 34 redundant variables, an average of $.34$ per equation, are included.

Table \ref{Table1} also reveals that the adaptive LASSO using ridge often outperforms the adaptive LASSO using the LASSO as first step estimator, albeit by a narrow margin. The post LASSO estimator results in a decrease of the estimation and forecast errors for $T=100$ and $T=500$, whereas for $T=50$ the post LASSO worsens slightly both errors.

This dimension reduction can result in a large reduction in RMSE compared to the least squares estimator including all variables. The LASSO and the adaptive LASSO are always more precise than this alternative. The adaptive LASSO tends to be more precise than the LASSO due to its more intelligent weights in the second step. However, it is still a little less precise than the oracle estimator -- a result which stems from the occasional inclusion of irrelevant variables\footnote{We also experimented with using least squares including all variables as initial estimator. However, it did not uniformly dominate the LASSO while being infeasible in settings with fewer observations than variables. Detailed computation results are available in the supplementary material.}. The two shrinkage procedures forecast as precisely as the oracle estimator except for the most difficult settings. As a consequence, they are more precise than least squares including all variables.

\begin{figure}
\begin{center}
\includegraphics[scale=.75]{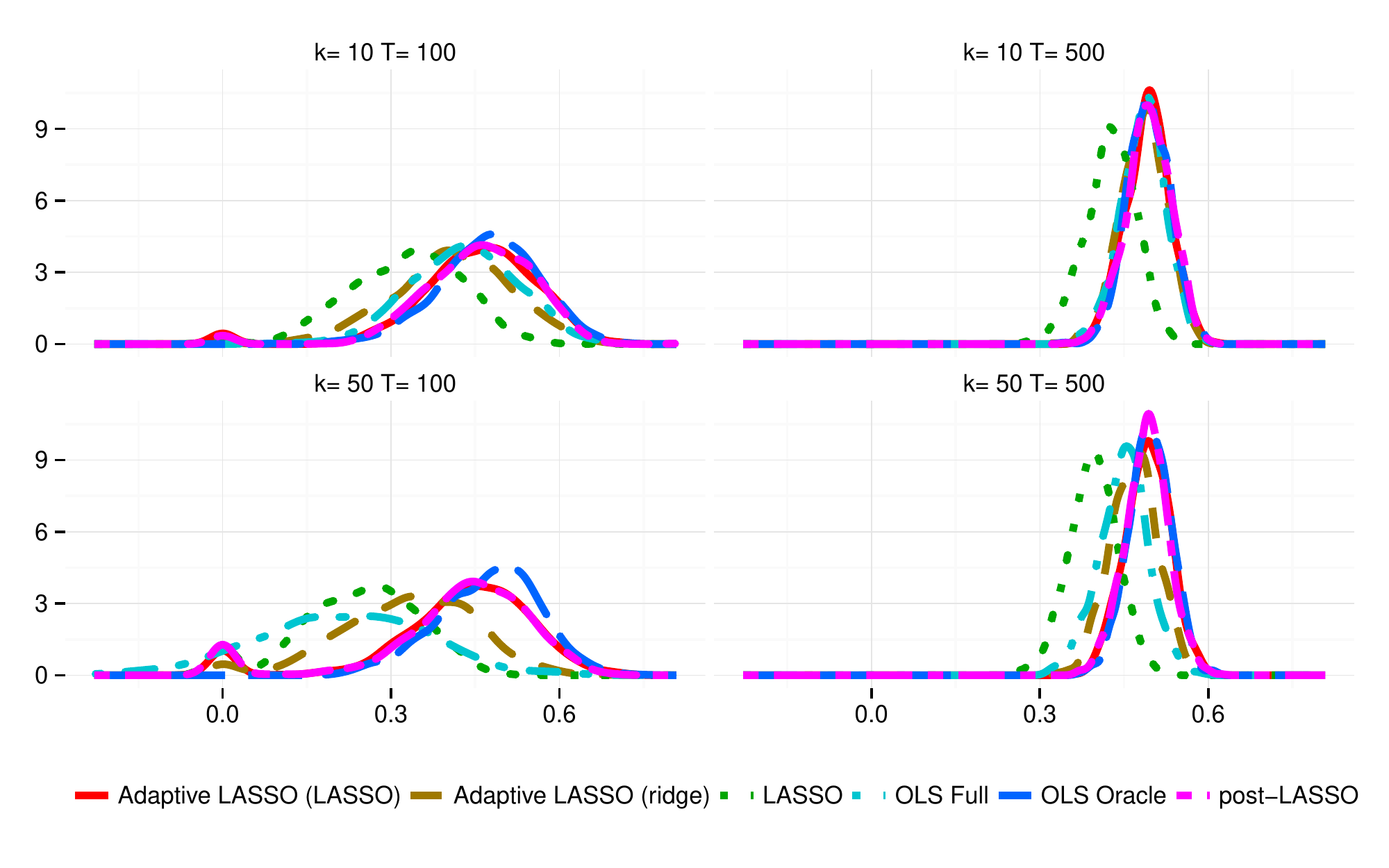}\vspace{-0.5cm}
\caption{\small Density of the estimates of the first parameter in the first equation. The true value of the parameter is $.5$.}
\label{Dens_beta}
\end{center}
\end{figure}

Figure \ref{Dens_beta} contains the densities of the estimates over the 1000 Monte Carlo replications of the first parameter in the first equation. The true value of this parameter is $.5$. The upper two plots are for $k=10$ and reveal that all procedures except for the LASSO (and to a lesser extend the adaptive LASSO with ridge as initial estimator) are centered at the right place. The LASSO is centered too far to the left due to its shrinkage, the post LASSO estimator corrects this bias.  

The bottom two plots are concerned with a high dimensional setting where $k=50$. Results for $k=100$ are not reported since least squares including all variables is only applicable for $T=500$ here. Two things are observed for $T=100$ when $k=50$. First, the least squares estimator including all variables has a very big variance and is not even centered the correct place. The adaptive LASSO does not suffer from this problem and is only slightly downwards biased compared to the least squares oracle. However, (and secondly) the LASSO and the adaptive LASSO have bimodal densities due to the occasional wrong exclusion of the non-zero parameter. Increasing the sample size to $500$ eliminates this problem and now the density of the adaptive LASSO sits almost on top of the one of the least squares oracle while the LASSO and full least squares procedures are still biased to the left.   

Table \ref{Table2} contains the results for Experiment B. This setting is more difficult than the one in in Experiment A since the model is less sparse and the system possesses a root close to the unit circle.

Notice that neither the LASSO nor the adaptive LASSO ever find exactly the true model. Both procedures leave out relevant variables even for $T=500$. However, the fraction of relevant variables included tends to be increasing in the sample size. The adaptive LASSO, irrespective of which initial estimator is used, discards relevant variables in the second step. This results in a situation (for $T=500$) where the number of variables included by the LASSO tends to be slightly larger than the ideal one while the opposite is the case for the adaptive LASSO.

As in Experiment A the LASSO as well as the adaptive LASSO have much lower RMSE than OLS including all covariates. Furthermore, the LASSO is now slightly more precise than the adaptive LASSO (using the LASSO as initial estimator). This finding is due to the fact that the LASSO tends to discard slightly fewer relevant variables than the adaptive LASSO. The LASSO is actually almost as precise as Oracle OLS for $T=500$. This shows that the adaptive LASSO need not always be more precise than the plain LASSO since in the second step estimation there is not only an opportunity to reduce the bias of the non-zero coefficients but also the possibility of wrongly excluding relevant variables.

	\begin{sidewaystable}
\resizebox{\linewidth}{!}{%
	\begin{tabular}{l rrr rrrrrr rrr rrr rrr}
\toprule
& \multicolumn{ 3 }{c}{ LASSO }& \multicolumn{ 3 }{c}{ post-LASSO }& \multicolumn{ 3 }{c}{ adaptive LASSO }& \multicolumn{ 3 }{c}{ adaptive LASSO }& \multicolumn{ 3 }{c}{ Oracle OLS }& \multicolumn{ 3 }{c}{Full  OLS }\\
&\multicolumn{ 6 }{c}{ } &\multicolumn{ 3 }{c}{ $1^{st}$ step: LASSO} &\multicolumn{ 3 }{c}{ $1^{st}$ step: Ridge } &\multicolumn{ 6 }{c}{  } \\
\cmidrule(r){ 2 - 4 }\cmidrule(r){ 5 - 7 }\cmidrule(r){ 8 - 10 }\cmidrule(r){ 11 - 13 }\cmidrule(r){ 14 - 16 }\cmidrule(r){ 17 - 19 }T &50 & 100 & 500 & 50 & 100 & 500 & 50 & 100 & 500 & 50 & 100 & 500 & 50 & 100 & 500 & 50 & 100 & 500\\
\midrule
k&\multicolumn{ 18 }{c}{ True model uncovered. }\\
\midrule
 10 & 0.00 & 0.00 & 0.00 & 0.00 & 0.00 & 0.00 & 0.00 & 0.00 & 0.00 & 0.00 & 0.00 & 0.00 & 1.00 & 1.00 & 1.00 & 0.00 & 0.00 & 0.00 \\ 
  20 & 0.00 & 0.00 & 0.00 &  & 0.00 & 0.00 & 0.00 & 0.00 & 0.00 & 0.00 & 0.00 & 0.00 & 1.00 & 1.00 & 1.00 &  & 0.00 & 0.00 \\ 
  50 & 0.00 & 0.00 & 0.00 &  &  & 0.00 & 0.00 & 0.00 & 0.00 & 0.00 & 0.00 & 0.00 & 1.00 & 1.00 & 1.00 &  &  & 0.00 \\ 
  \midrule
&\multicolumn{ 18 }{c}{ True model included. }\\
\midrule
 10 & 0.00 & 0.00 & 0.22 & 0.00 & 0.00 & 0.22 & 0.00 & 0.00 & 0.00 & 0.00 & 0.00 & 0.00 & 1.00 & 1.00 & 1.00 & 1.00 & 1.00 & 1.00 \\ 
  20 & 0.00 & 0.00 & 0.02 &  & 0.00 & 0.02 & 0.00 & 0.00 & 0.00 & 0.00 & 0.00 & 0.00 & 1.00 & 1.00 & 1.00 &  & 1.00 & 1.00 \\ 
  50 & 0.00 & 0.00 & 0.00 &  &  & 0.00 & 0.00 & 0.00 & 0.00 & 0.00 & 0.00 & 0.00 & 1.00 & 1.00 & 1.00 &  &  & 1.00 \\ 
  \midrule
&\multicolumn{ 18 }{c}{ Share of relevant variables selected. }\\
\midrule
 10 & 0.46 & 0.61 & 0.98 & 0.46 & 0.61 & 0.98 & 0.36 & 0.43 & 0.88 & 0.42 & 0.56 & 0.86 & 1.00 & 1.00 & 1.00 & 1.00 & 1.00 & 1.00 \\ 
  20 & 0.59 & 0.52 & 0.98 &  & 0.52 & 0.98 & 0.40 & 0.37 & 0.87 & 0.37 & 0.50 & 0.87 & 1.00 & 1.00 & 1.00 &  & 1.00 & 1.00 \\ 
  50 & 0.33 & 0.55 & 0.98 &  &  & 0.98 & 0.21 & 0.38 & 0.87 & 0.37 & 0.39 & 0.96 & 1.00 & 1.00 & 1.00 &  &  & 1.00 \\ 
  \midrule
&\multicolumn{ 18 }{c}{ Number of variables selected. }\\
\midrule
 10 & 129 &  81 & 109 & 129 &  81 & 109 & 102 &  51 &  90 & 103 &  69 &  94 & 100 & 100 & 100 & 400 & 400 & 400 \\ 
  20 & 841 & 174 & 231 &  & 174 & 231 & 505 & 109 & 184 & 293 & 152 & 195 & 200 & 200 & 200 &  & 1600 & 1600 \\ 
  50 & 2307 & 3011 & 623 &  &  & 623 & 1244 & 1604 & 476 & 1837 & 432 & 538 & 500 & 500 & 500 &  &  & 10000 \\ 
  \midrule
&\multicolumn{ 18 }{c}{ Root mean square estimation error. }\\
\midrule
 10 & 4.38 & 1.01 & 0.45 & 4.91 & 1.24 & 0.48 & 4.26 & 1.34 & 0.59 & 3.63 & 1.08 & 0.66 & 1.58 & 1.00 & 0.41 & 9.89 & 2.74 & 0.94 \\ 
  20 & 5.23 & 1.57 & 0.67 &  & 1.95 & 0.71 & 5.16 & 2.03 & 0.85 & 3.27 & 1.60 & 0.94 & 2.23 & 1.41 & 0.58 &  & 10.82 & 1.98 \\ 
  50 & 6.20 & 6.03 & 1.12 &  &  & 1.21 & 6.64 & 6.09 & 1.39 & 6.29 & 2.79 & 1.10 & 3.52 & 2.23 & 0.92 &  &  & 5.87 \\ 
  \midrule
&\multicolumn{ 18 }{c}{ Root mean square 1-step ahead forecast error. }\\
\midrule
 10 & 0.183 & 0.114 & 0.101 & 0.195 & 0.113 & 0.100 & 0.176 & 0.114 & 0.102 & 0.161 & 0.114 & 0.103 & 0.114 & 0.108 & 0.102 & 0.338 & 0.132 & 0.104 \\ 
  20 & 0.163 & 0.117 & 0.103 &  & 0.116 & 0.102 & 0.164 & 0.118 & 0.103 & 0.135 & 0.115 & 0.103 & 0.116 & 0.107 & 0.101 &  & 0.269 & 0.108 \\ 
  50 & 0.143 & 0.141 & 0.104 &  &  & 0.103 & 0.147 & 0.139 & 0.102 & 0.144 & 0.120 & 0.103 & 0.116 & 0.106 & 0.101 &  &  & 0.131 \\ 
  \bottomrule

	\end{tabular}
}
\caption{\small The results for Experiment B measured along the dimensions discussed in the main text.}
\label{Table2}
	\end{sidewaystable}

The adaptive LASSO using Ridge as a first step outperforms the adaptive LASSO using the LASSO as a first step in small samples ($T=50,100$). However in the instances where the LASSO performs almost as well as Oracle OLS, it provides more accurate weights for the adaptive LASSO than Ridge does leading to better performances for the adaptive LASSO using LASSO in the first step. The post-LASSO estimator does not improve on the LASSO in this experiment. This could be due to the fact that the first step LASSO has excluded relevant variables such that the second stage least squares estimation takes place in a misspecified model.

Table \ref{Table3} contains the results for Experiment C. As was the case in Experiment B, neither the LASSO nor the adaptive LASSO unveil the true model. However, they tend to at least retain the relevant variables as the sample size increases and the share of relevant variables is also always above 90 percent when $T=500$. As in Experiment A, the adaptive LASSO does not discard many relevant variables in the second estimation step. In fact, turning to the number of variables selected, this second step is very useful since it often greatly reduces the number of irrelevant variables included by the LASSO in the first step. Put differently, the LASSO carries out the rough initial screening in the first step while the adaptive LASSO fine tunes this in the second step.

The adaptive LASSO always estimates the parameters more precisely than full OLS (and is also more precise than the LASSO for $T=500$). As in the previous experiments this results in forecasts that are as precise as the OLS oracle for $T=500$. Note that in this experiment the adaptive LASSO using the LASSO as a first step estimator is more precise for $T=500$ than its counterpart using ridge regression. The evidence is more mixed for smaller sample sizes.

	\begin{sidewaystable}
\resizebox{\linewidth}{!}{%
	\begin{tabular}{l rrr rrr rrr rrr rrr rrr}
	
\toprule
& \multicolumn{ 3 }{c}{ LASSO }& \multicolumn{ 3 }{c}{ post-LASSO }& \multicolumn{ 3 }{c}{ adaptive LASSO  }& \multicolumn{ 3 }{c}{ adaptive LASSO  }& \multicolumn{ 3 }{c}{ Oracle OLS }& \multicolumn{ 3 }{c}{Full  OLS }\\
&\multicolumn{ 6 }{c}{ } &\multicolumn{ 3 }{c}{ $1^{st}$ step: LASSO} &\multicolumn{ 3 }{c}{ $1^{st}$ step: Ridge } &\multicolumn{ 6 }{c}{  } \\
\cmidrule(r){ 2 - 4 }\cmidrule(r){ 5 - 7 }\cmidrule(r){ 8 - 10 }\cmidrule(r){ 11 - 13 }\cmidrule(r){ 14 - 16 }\cmidrule(r){ 17 - 19 }T &50 & 100 & 500 & 50 & 100 & 500 & 50 & 100 & 500 & 50 & 100 & 500 & 50 & 100 & 500 & 50 & 100 & 500\\
\midrule
k&\multicolumn{ 18 }{c}{ True model uncovered. }\\
\midrule
 10 & 0.00 & 0.00 & 0.00 & 0.00 & 0.00 & 0.00 & 0.00 & 0.00 & 0.01 & 0.00 & 0.00 & 0.00 & 1.00 & 1.00 & 1.00 &  & 0.00 & 0.00 \\ 
  20 & 0.00 & 0.00 & 0.00 &  & 0.00 & 0.00 & 0.00 & 0.00 & 0.00 & 0.00 & 0.00 & 0.00 & 1.00 & 1.00 & 1.00 &  &  & 0.00 \\ 
  50 & 0.00 & 0.00 & 0.00 &  &  & 0.00 & 0.00 & 0.00 & 0.00 & 0.00 & 0.00 & 0.00 & 1.00 & 1.00 & 1.00 &  &  & 0.00 \\ 
  \midrule
&\multicolumn{ 18 }{c}{ True model included. }\\
\midrule
 10 & 0.00 & 0.00 & 1.00 & 0.00 & 0.00 & 1.00 & 0.00 & 0.00 & 1.00 & 0.00 & 0.00 & 1.00 & 1.00 & 1.00 & 1.00 &  & 1.00 & 1.00 \\ 
  20 & 0.00 & 0.00 & 0.96 &  & 0.00 & 0.96 & 0.00 & 0.00 & 0.96 & 0.00 & 0.00 & 0.94 & 1.00 & 1.00 & 1.00 &  &  & 1.00 \\ 
  50 & 0.00 & 0.00 & 0.00 &  &  & 0.00 & 0.00 & 0.00 & 0.00 & 0.00 & 0.00 & 0.00 & 1.00 & 1.00 & 1.00 &  &  & 1.00 \\ 
  \midrule
&\multicolumn{ 18 }{c}{ Share of relevant variables selected. }\\
\midrule
 10 & 0.69 & 0.62 & 1.00 & 0.69 & 0.62 & 1.00 & 0.51 & 0.59 & 1.00 & 0.45 & 0.59 & 1.00 & 1.00 & 1.00 & 1.00 &  & 1.00 & 1.00 \\ 
  20 & 0.43 & 0.57 & 1.00 &  & 0.57 & 1.00 & 0.28 & 0.49 & 1.00 & 0.40 & 0.44 & 1.00 & 1.00 & 1.00 & 1.00 &  &  & 1.00 \\ 
  50 & 0.19 & 0.46 & 0.93 &  &  & 0.93 & 0.11 & 0.31 & 0.93 & 0.22 & 0.39 & 0.78 & 1.00 & 1.00 & 1.00 &  &  & 1.00 \\ 
  \midrule
&\multicolumn{ 18 }{c}{ Number of variables selected. }\\
\midrule
 10 & 307 & 144 & 199 & 307 & 144 & 199 & 199 &  95 &  55 & 142 & 118 &  59 &  50 &  50 &  50 &  & 500 & 500 \\ 
  20 & 741 & 660 & 561 &  & 660 & 561 & 434 & 406 & 122 & 509 & 288 & 140 & 100 & 100 & 100 &  &  & 2000 \\ 
  50 & 1958 & 3823 & 1762 &  &  & 1762 & 1041 & 2039 & 407 & 1546 & 1577 & 1362 & 250 & 250 & 250 &  &  & 12500 \\ 
  \midrule
&\multicolumn{ 18 }{c}{ Root mean square estimation error. }\\
\midrule
 10 & 6.26 & 4.81 & 1.66 & 7.18 & 4.57 & 0.71 & 6.20 & 4.47 & 0.48 & 5.59 & 4.70 & 0.67 & 1.24 & 0.76 & 0.30 &  & 4.67 & 1.16 \\ 
  20 & 8.85 & 7.80 & 3.25 &  & 8.26 & 1.45 & 8.97 & 7.58 & 0.84 & 8.50 & 7.45 & 1.32 & 1.75 & 1.07 & 0.42 &  &  & 2.78 \\ 
  50 & 13.75 & 13.60 & 7.83 &  &  & 5.30 & 14.00 & 13.72 & 4.42 & 13.64 & 12.66 & 8.32 & 2.77 & 1.70 & 0.66 &  &  & 10.03 \\ 
  \midrule
&\multicolumn{ 18 }{c}{ Root mean square 1-step ahead forecast error. }\\
\midrule
 10 & 0.173 & 0.142 & 0.107 & 0.204 & 0.145 & 0.106 & 0.171 & 0.140 & 0.102 & 0.153 & 0.140 & 0.103 & 0.109 & 0.103 & 0.101 &  & 0.160 & 0.107 \\ 
  20 & 0.143 & 0.149 & 0.111 &  & 0.163 & 0.108 & 0.150 & 0.151 & 0.102 & 0.148 & 0.144 & 0.102 & 0.108 & 0.103 & 0.102 &  &  & 0.117 \\ 
  50 & 0.128 & 0.137 & 0.121 &  &  & 0.116 & 0.135 & 0.144 & 0.109 & 0.134 & 0.142 & 0.126 & 0.108 & 0.104 & 0.100 &  &  & 0.153 \\ 
  \bottomrule
					
	\end{tabular}
}
\caption{\small The results for Experiment C measured along the dimensions discussed in the main text.}
\label{Table3}
	\end{sidewaystable}

Finally, Table \ref{Table4} contains the results for Experiment D. In this setting no parameters are zero so it is sensible that the shrinkage procedures do not fare well in unveiling or including the true model. Also, they always include less than half of the relevant variables and considerably less for large $k$. However, in terms of estimation error \textit{all} shrinkage estimators are more precise than the oracle irrespective of sample size or number of variables in the model. This can be explained by the many very small but non-zero parameters in the model. Shrinking these parameters turns out to be more precise than estimating them unrestricted as done by the least squares oracle\footnote{As suggested by an anonymous referee we also investigated the properties of an "oracle" which knows the true paramters and calculates a linear combination of the covariates accordingly. The only thing which this oracle estimates is the coefficient to this linear combination (ideally this coeffient is one). This oracle yielded superior parameter estimates to the shrinkage estimators but did not forecast much more precisely.}. This gain in precision manifests itself in the shrinkage estimators delivering uniformly more precise forecasts. These findings are encouraging since the LASSO-type estimators perform well even when some of the assumptions underpinning the theoretical results are violated.

	\begin{sidewaystable}

\resizebox{\linewidth}{!}{%
	\begin{tabular}{l rrr rrr rrr rrr rrr rrr}
	
\toprule
& \multicolumn{ 3 }{c}{ LASSO }& \multicolumn{ 3 }{c}{ post-LASSO }& \multicolumn{ 3 }{c}{ adaptive LASSO }& \multicolumn{ 3 }{c}{ adaptive LASSO }& \multicolumn{ 3 }{c}{ Oracle OLS }& \multicolumn{ 3 }{c}{Full  OLS }\\
&\multicolumn{ 6 }{c}{ } &\multicolumn{ 3 }{c}{ $1^{st}$ step: LASSO} &\multicolumn{ 3 }{c}{ $1^{st}$ step: Ridge } &\multicolumn{ 6 }{c}{  } \\
\cmidrule(r){ 2 - 4 }\cmidrule(r){ 5 - 7 }\cmidrule(r){ 8 - 10 }\cmidrule(r){ 11 - 13 }\cmidrule(r){ 14 - 16 }\cmidrule(r){ 17 - 19 }T &50 & 100 & 500 & 50 & 100 & 500 & 50 & 100 & 500 & 50 & 100 & 500 & 50 & 100 & 500 & 50 & 100 & 500\\
\midrule
k&\multicolumn{ 18 }{c}{ True model uncovered. }\\
\midrule
 10 & 0.00 & 0.00 & 0.00 & 0.00 & 0.00 & 0.00 & 0.00 & 0.00 & 0.00 & 0.00 & 0.00 & 0.00 & 1.00 & 1.00 & 1.00 & 1.00 & 1.00 & 1.00 \\ 
  20 & 0.00 & 0.00 & 0.00 & 0.00 & 0.00 & 0.00 & 0.00 & 0.00 & 0.00 & 0.00 & 0.00 & 0.00 & 1.00 & 1.00 & 1.00 & 1.00 & 1.00 & 1.00 \\ 
  50 & 0.00 & 0.00 & 0.00 &  & 0.00 & 0.00 & 0.00 & 0.00 & 0.00 & 0.00 & 0.00 & 0.00 &  & 1.00 & 1.00 &  & 1.00 & 1.00 \\ 
  \midrule
&\multicolumn{ 18 }{c}{ True model included. }\\
\midrule
 10 & 0.00 & 0.00 & 0.00 & 0.00 & 0.00 & 0.00 & 0.00 & 0.00 & 0.00 & 0.00 & 0.00 & 0.00 & 1.00 & 1.00 & 1.00 & 1.00 & 1.00 & 1.00 \\ 
  20 & 0.00 & 0.00 & 0.00 & 0.00 & 0.00 & 0.00 & 0.00 & 0.00 & 0.00 & 0.00 & 0.00 & 0.00 & 1.00 & 1.00 & 1.00 & 1.00 & 1.00 & 1.00 \\ 
  50 & 0.00 & 0.00 & 0.00 &  & 0.00 & 0.00 & 0.00 & 0.00 & 0.00 & 0.00 & 0.00 & 0.00 &  & 1.00 & 1.00 &  & 1.00 & 1.00 \\ 
  \midrule
&\multicolumn{ 18 }{c}{ Share of relevant variables selected. }\\
\midrule
 10 & 0.22 & 0.32 & 0.48 & 0.22 & 0.32 & 0.48 & 0.17 & 0.24 & 0.36 & 0.21 & 0.26 & 0.36 & 1.00 & 1.00 & 1.00 & 1.00 & 1.00 & 1.00 \\ 
  20 & 0.13 & 0.17 & 0.27 & 0.13 & 0.17 & 0.27 & 0.10 & 0.13 & 0.20 & 0.14 & 0.15 & 0.20 & 1.00 & 1.00 & 1.00 & 1.00 & 1.00 & 1.00 \\ 
  50 & 0.36 & 0.07 & 0.11 &  & 0.07 & 0.11 & 0.25 & 0.06 & 0.08 & 0.09 & 0.07 & 0.09 &  & 1.00 & 1.00 &  & 1.00 & 1.00 \\ 
  \midrule
&\multicolumn{ 18 }{c}{ Number of variables selected. }\\
\midrule
 10 &  22 &  31 &  48 &  22 &  31 &  48 &  17 &  23 &  35 &  20 &  25 &  36 & 100 & 100 & 100 & 100 & 100 & 100 \\ 
  20 &  52 &  67 & 106 &  52 &  67 & 106 &  41 &  53 &  78 &  54 &  60 &  79 & 400 & 400 & 400 & 400 & 400 & 400 \\ 
  50 & 903 & 180 & 273 &  & 180 & 273 & 622 & 149 & 211 & 226 & 175 & 213 &  & 2500 & 2500 &  & 2500 & 2500 \\ 
  \midrule
&\multicolumn{ 18 }{c}{ Root mean square estimation error. }\\
\midrule
 10 & 1.18 & 0.82 & 0.36 & 1.30 & 0.86 & 0.35 & 1.29 & 0.86 & 0.37 & 1.18 & 0.82 & 0.37 & 1.63 & 1.02 & 0.41 & 1.63 & 1.02 & 0.41 \\ 
  20 & 1.82 & 1.27 & 0.54 & 2.14 & 1.39 & 0.56 & 2.08 & 1.35 & 0.56 & 1.82 & 1.21 & 0.56 & 3.93 & 2.22 & 0.83 & 3.93 & 2.22 & 0.83 \\ 
  50 & 7.82 & 2.19 & 0.91 &  & 2.52 & 0.97 & 7.28 & 2.45 & 0.95 & 3.47 & 2.06 & 0.92 &  & 7.40 & 2.19 &  & 7.40 & 2.19 \\ 
  \midrule
&\multicolumn{ 18 }{c}{ Root mean square 1-step ahead forecast error. }\\
\midrule
 10 & 0.112 & 0.105 & 0.101 & 0.113 & 0.105 & 0.101 & 0.113 & 0.106 & 0.101 & 0.111 & 0.107 & 0.101 & 0.115 & 0.107 & 0.101 & 0.115 & 0.107 & 0.101 \\ 
  20 & 0.114 & 0.108 & 0.102 & 0.116 & 0.106 & 0.101 & 0.117 & 0.106 & 0.101 & 0.113 & 0.106 & 0.101 & 0.136 & 0.114 & 0.103 & 0.136 & 0.114 & 0.103 \\ 
  50 & 0.154 & 0.111 & 0.102 &  & 0.109 & 0.101 & 0.150 & 0.110 & 0.101 & 0.116 & 0.108 & 0.102 &  & 0.148 & 0.106 &  & 0.148 & 0.106 \\ 
  \bottomrule
	\end{tabular}
}
\caption{\small The results for Experiment D measured along the dimensions discussed in the main text.}
	\label{Table4}
	\end{sidewaystable}

\section{Conclusions}\label{Conclusions}
This paper is concerned with estimation of high-dimensional stationary vector autoregressions. In particular, the focus is on the LASSO and the adaptive LASSO. We establish upper bounds for the prediction and estimation error of the LASSO. The novelty in these upper bounds is that they are non-asymptotic. Under further conditions it is shown that all relevant variables are retained with high probability. A comparison to oracle assisted least squares is made and it is seen that the LASSO does not perform much worse than this infeasible procedure. The finite sample results are then used to establish equivalent asymptotic results. It is seen that the LASSO is consistent even when the number of parameters grows sub-exponentially with the sample size.

Next, lower bounds on the probability with which the adaptive LASSO unveils the correct sign pattern are given. Again these results are non-asymptotic but they can be used to establish asymptotic sign consistency of the adaptive LASSO. As for the LASSO the number of parameters is allowed to grow sub-exponentially fast with the sample size. Finally, we show that the estimates of the non-zero coefficients are asymptotically equivalent to those obtained by least squares applied to the model only including the relevant covariates.  

The main technical novelty in the above results is the handling of the restricted eigenvalue condition in high-dimensional systems with dependent covariates. In particular, a finite sample bound on the estimation error of the empirical covariance matrix is established to this end.

We believe that these results may be useful for the applied researcher who often faces the curse of dimensionality when building VAR models since the number of parameters increases quadratically with the number of variables included. However, the LASSO and the adaptive LASSO are applicable even in these situations. 

In future research it may be interesting to further investigate the possibilities of handling empirically relevant models like the invertible $MA(1)$ which is currently not covered by our framework.

Furthermore, it is of interest to derive a theoretically justified data-driven method for choosing $\lambda_T$. However, we defer this to future work at this stage.

Finally, this paper has been concerned with stationary vector autoregressions and it is of interest to investigate whether similar oracle inequalities may hold for non-stationary VARs.




\section{Appendix}
We start by stating a couple of preparatory lemmas. The first lemma bounds the probability of the maximum of all possible cross terms between explanatory variables and error terms becoming large. This bound will be used in the proof of Lemma \ref{Bound} below.

\begin{lemma}\label{Orlicz}
Let Assumption \ref{Ass1} be satisfied. Then, for any $L_T>0$,
\begin{align*}
P\left(\max_{1\leq t\leq T}\max_{1\leq i\leq k}\max_{1\leq l\leq p}\max_{1\leq j\leq k}|y_{t-l,i}\epsilon_{t,j}|\geq L_T\right)
\leq 2\exp\left(\frac{-L_T}{A\ln(1+T)\ln(1+k)^2\ln(1+p)\sigma_T^2}\right)
\end{align*}
for some positive constant $A$.
\end{lemma}

In order to prove Lemma \ref{Orlicz} Orlicz norms turn out to be useful since random variables with bounded Orlicz norms obey useful maximal inequalities. Let $\psi$ be a non-decreasing convex function with $\psi(0)=0$. Then, the Orlicz norm of a random variable $X$ is given by

\begin{align*}
\enVert{X}_\psi=\inf\left\{C>0:E\psi\left(|X|/C\right)\leq 1\right\}
\end{align*}

where, as usual, $\inf \emptyset =\infty$. By choosing $\psi(x)=x^p$ the Orlicz norm reduces to the usual $L^p$-norm since for $X\in L^p$, $C$ equals $E(|X|^p)^{1/p}$. However, for our purpose $\psi(x)=e^x-1$. One has the following maximal inequality:

\begin{lemma}[Lemma 2.2.2 from \cite{vdVW96}]\label{OLemma}
Let $\psi(x)$ be a convex, non-decreasing, non-zero function with $\psi(0)=0$ and $\lim\sup_{x,y\rightarrow\infty}\psi(x)\psi(y)/\psi(cxy)<\infty$ for some constant $c$. Then for any random variables, $X_1,...,X_m$,
\begin{align*}
\enVert[3]{\max_{1\leq i\leq m}X_i}_{\psi}\leq K\psi^{-1}(m)\max_{1\leq i\leq m}\enVert{X_i}_\psi
\end{align*}
for a constant K depending only on $\psi$.
\end{lemma}
Notice that this result is particularly useful if $\psi^{-1}(x)$ only increases slowly which is the case when $\psi(x)$ increases very fast as in our case.

\begin{proof}[Proof of Lemma \ref{Orlicz}]
Let $\psi(x)=e^x-1$. First we show that\\ $\enVert{\max_{1\leq t\leq T}\max_{1\leq i\leq k}\max_{1\leq l\leq p}\max_{1\leq j\leq k}y_{t-l,i}\epsilon_{t,j}}_\psi<\infty$. Repeated application of Lemma \ref{OLemma} yields
\begin{align}
&\enVert{\max_{1\leq t\leq T}\max_{1\leq i\leq k}\max_{1\leq l\leq p}\max_{1\leq j\leq k}y_{t-l,i}\epsilon_{t,j}}_\psi\notag\\
&\leq 
K^4 \ln(1+T)\ln(1+k)^2\ln(1+p) \max_{1\leq t\leq T}\max_{1\leq i\leq k}\max_{1\leq l\leq p}\max_{1\leq j\leq k} \enVert{y_{t-l,i}\epsilon_{t,j}}_{\psi}\label{ROrlicz}
\end{align}
Next, we turn to bounding $\enVert{y_{t-l,i}\epsilon_{t,j}}_{\psi}$ uniformly in $1\leq i,j\leq k,\ 1\leq l \leq p$ and $1\leq t\leq T$. Since $y_{t-l,i}$ and $\epsilon_{t,j}$ are both gaussian with mean 0 and variances $\sigma^2_{i,y}$ and $\sigma^2_{j,\epsilon}$ respectively it follows by a standard estimate on gaussian tails (see e.g. \cite{bill99}, page 263) that for any $x>0$

\begin{align*}
P\left(|y_{t-l,i}\epsilon_{t,j}|>x\right)&
\leq P\left(|y_{t-l,i}|>\sqrt{x}\right)+P\left(|\epsilon_{t,i}|>\sqrt{x}\right)\leq 2e^{-x/2\sigma_{i,y}^2}+2e^{-x/2\sigma_{j,\epsilon}^2}\\
&\leq 4e^{\frac{-x}{2\sigma_T^2}}
\end{align*} 

Hence, $\left\{y_{t-l,i}\epsilon_{t,j}\right\}$ has subexponential tails\footnote{A random variable $X$ is said to have subexponential tails if there exists constants $K$ and $C$ such that for every $x>0$, $P(|X|>x)\leq Ke^{-Cx}$.} and it follows from Lemma 2.2.1 in \cite{vdVW96} that $\enVert[0]{y_{t-l,i}\epsilon_{t,j}}_\psi\leq 10\sigma_T^2$. Using this in (\ref{ROrlicz}) yields
\begin{align*}
\enVert{\max_{1\leq t\leq T}\max_{1\leq i\leq k}\max_{1\leq l\leq p}\max_{1\leq j\leq k}y_{t-l,i}\epsilon_{t,j}}_\psi
&\leq K^4 \ln(1+T)\ln(1+k)^2\ln(1+p) 10\sigma_T^2\\
&=A\ln(1+T)\ln(1+k)^2\ln(1+p)\sigma_T^2:=f(T)
\end{align*}
where $A:=10K^4$. Finally, by Markov's inequality, the definition of the Orlicz norm, and the fact that $1\wedge \psi(x)^{-1}=1\wedge (e^x-1)^{-1}\leq 2e^{-x}$,

\begin{align*}
&P\left(\max_{1\leq t\leq T}\max_{1\leq i\leq k}\max_{1\leq l\leq p}\max_{1\leq j\leq k}|y_{t-l,i}\epsilon_{t,j}|\geq L_T\right)\\
&=P\left(\psi\left(\max_{1\leq t\leq T}\max_{1\leq i\leq k}\max_{1\leq l\leq p}\max_{1\leq j\leq k}|y_{t-l,i}\epsilon_{t,j}|/f(T)\right)\geq \psi\left(L_T/f(T)\right)\right)\\
& \leq
1\wedge\frac{E\psi\left(\max_{1\leq t\leq T}\max_{1\leq i\leq k}\max_{1\leq l\leq p}\max_{1\leq j\leq k}|y_{t-l,i}\epsilon_{t,j}|/f(T)\right)}{\psi\left[L_T/f(T)\right]}\\
&\leq
1\wedge\frac{1}{\psi\left[L_T/f(T)\right]}\\
&\leq 2\exp\left(-L_T/f(T)\right)\\
&=2\exp(-L_T/[A\ln(1+T)\ln(1+k)^2\ln(1+p)\sigma_T^2])
\end{align*}
\end{proof}

\begin{lemma}\label{Bound}
Let Assumption \ref{Ass1} be satisfied and define
\begin{align}
\mathcal{B}_T=\cbr[3]{\max_{1\leq i\leq k}\max_{1\leq l\leq p}\max_{1\leq j\leq k}\envert[3]{\frac{1}{T}\sum_{t=1}^Ty_{t-l,i}\epsilon_{t,j}}< \frac{\lambda_T}{2}}\label{BT}
\end{align}
Then,

\begin{align}
P(\mathcal{B}_T)\geq 1-2(k^2p)^{1-\ln(1+T)}-2(1+T)^{-A}\label{BTprob}
\end{align}
for $\lambda_T=\sqrt{8\ln(1+T)^5\ln(1+k)^4\ln(1+p)^2\ln(k^2p)\sigma_T^4/T}$ and $A$ a positive constant.

\end{lemma}

In order to prove Lemma \ref{Bound} we need the following result, the so-called \textit{useful rule}, on conditional expectations adapted from \cite{hoffmann94} for our purpose.

\begin{lemma}[(6.8.14) in \cite{hoffmann94}]\label{hoffmann94}
Let $f:\mathbb{R}\times\mathbb{R}\to\mathbb{R}$ be measurable such that $|f(U,V)|$ is integrable and $f(U,v)$ is integrable for $P_V$ almost all $v\in \mathbb{R}$ (here $P_V$ denotes the distribution of $V$), and let $\phi(v)=E(f(U,v))$. If, for a sigma field $\mathcal{G}$, $V$ is measurable with respect to $\mathcal{G}$ and U is independent of $\mathcal{G}$, then we have
\begin{align*}
E(f(U,V)|\mathcal{G})=\phi(V)\text{ $P$-almost surely}
\end{align*}  
\end{lemma} 

\begin{proof}[Proof of Lemma \ref{Bound}]
By subadditivity of the probability measure it follows that for any $L_T>0$,

\begin{align*}
&P\left(\max_{1\leq i\leq k}\max_{1\leq l\leq p}\max_{1\leq j\leq k}\envert[3]{\frac{1}{T}\sum_{t=1}^Ty_{t-l,i}\epsilon_{t,j}}\geq \frac{\lambda_T}{2}\right)\\
&=P\left(\bigcup_{i=1}^k\bigcup_{l=1}^p\bigcup_{j=1}^k\left\{ \envert[3]{\frac{1}{T}\sum_{t=1}^Ty_{t-l,i}\epsilon_{t,j}}\geq \frac{\lambda_T}{2}\right\}\right)\\
&\leq 
P\left(\bigcup_{i=1}^k\bigcup_{l=1}^p\bigcup_{j=1}^k\left\{ \envert[3]{\frac{1}{T}\sum_{t=1}^Ty_{t-l,i}\epsilon_{t,j}}\geq \frac{\lambda_T}{2}\right\}\cap \bigcap_{t=1}^T \bigcap_{i=1}^k\bigcap_{l=1}^p\bigcap_{j=1}^k\left\{|y_{t-l,i}\epsilon_{t,j}|<L_T\right\}\right)\\
&+P\left(\left\{\bigcap_{t=1}^T \bigcap_{i=1}^k\bigcap_{l=1}^p\bigcap_{j=1}^k\left\{|y_{t-l,i}\epsilon_{t,j}|<L_T\right\}\right\}^c\right)\\
&\leq \sum_{i=1}^k\sum_{l=1}^p\sum_{j=1}^kP\left(\envert[3]{\frac{1}{T}\sum_{t=1}^Ty_{t-l,i}\epsilon_{t,j}}\geq \frac{\lambda_T}{2}, \bigcap_{t=1}^T\left\{|y_{t-l,i}\epsilon_{t,j}|<L_T\right\}\right)\\
&+P\left(\max_{1\leq t\leq T}\max_{1\leq i\leq k}\max_{1\leq l\leq p}\max_{1\leq j\leq k}|y_{t-l,i}\epsilon_{t,j}|\geq L_T\right)
\end{align*} 
Observe that  for $\mathcal{F}_t=\sigma\left(\left\{\epsilon_s,\ s=1,...,t; \ y_s,\ s=1,...,t\right\}\right)$ being the natural filtration $\left\{y_{t-l,i}\epsilon_{t,j}I_{\cbr[0]{|y_{t-l,i}\epsilon_{t,j}|< L_T}}, \mathcal{F}_t\right\}_{t=1}^\infty$ defines a martingale difference sequence for every $1\leq i,j\leq k$ and $1\leq l\leq p$ since
\begin{align*}
E(y_{t-l,i}\epsilon_{t,j}I_{\cbr[0]{|y_{t-l,i}\epsilon_{t,j}|< L_T}}|\mathcal{F}_{t-1})
=
y_{t-l,i}E(\epsilon_{t,j}I_{\cbr[0]{|y_{t-l,i}\epsilon_{t,j}|< L_T}}|\mathcal{F}_{t-1})
=
0
\end{align*}
where the second equality follows from Lemma \ref{hoffmann94} with $f(\epsilon_{t,j}, y_{t-l,i})=\epsilon_{t,j}I_{\cbr[0]{|y_{t-l,i}\epsilon_{t,j}|\leq L_T}}$ such that for all $v\in \mathbb{R}$ \footnote{The argument handles the case $v\neq 0$. The case $v=0$ follows from omitting the second to last equality and using $E(\epsilon_{t,j})=0$.}
\begin{align*}
\phi(v)=E(f(\epsilon_{t,j},v))
=
E(\epsilon_{t,j}I_{\cbr[0]{|v\epsilon_{t,j}|< L_T}})
=
E(\epsilon_{t,j}I_{\cbr[0]{|\epsilon_{t,j}|< L_T/|v|}})
=
0
\end{align*}
where the last equality follows from the gaussianity of the mean zero $\epsilon_{t,j}$.\footnote{More precisely, symmetrically truncating a symmetric mean zero variable yields a new variable with mean zero.} Hence, $E(\epsilon_{t,j}I_{\cbr[0]{|y_{t-l,i}\epsilon_{t,j}|\leq L_T}}|\mathcal{F}_{t-1})=\phi(y_{t-l,i})=0$ $P$-almost surely.  
Next, using the Azuma-Hoeffding inequality\footnote{The Azuma-Hoeffding inequality is now applicable since we apply it on the set where the summands are bounded by $L_T$.} on the first term and Lemma \ref{Orlicz} on the second term with $L_T=\ln(1+T)^2\ln(1+k)^2\ln(1+p)\sigma_T^2$ yields,

\begin{align*}
&P\left(\max_{1\leq i\leq k}\max_{1\leq l\leq p}\max_{1\leq j\leq k}\envert[3]{\frac{1}{T}\sum_{t=1}^Ty_{t-l,i}\epsilon_{t,j}}\geq \frac{\lambda_T}{2}\right)\\
&\leq 
k^2p\cdot2\exp\left(-\frac{T\lambda_T^2}{8L_T^2}\right)+2\exp\left(-\frac{\ln(1+T)}{A}\right)\\
&=
2k^2p\cdot \exp\left(-\ln(1+T)\ln(k^2p)\right)+2(1+T)^{-1/A}\\
&=2(k^2p)^{1-\ln(1+T)}+2(1+T)^{-1/A}
\end{align*}
\end{proof}

\begin{proof}[Proof of Theorem \ref{Thm1}]
As the results are equation by equation we shall focus on equation $i$ here but omit the subscript $i$ for brevity.
By the minimizing property of $\hat{\beta}$ it follows that
\begin{align*}
\frac{1}{T}\enVert[1]{y-X\hat{\beta}}^2+2\lambda_T\enVert[1]{\hat{\beta}}_{\ell_1}\leq 
\frac{1}{T}\enVert[1]{y-X\beta^*}^2+2\lambda_T\enVert[1]{\beta^*}_{\ell_1}
\end{align*}
which using that $y=X\beta^*+\epsilon$ yields
\begin{align*}
\frac{1}{T}\enVert{\epsilon}^2+\frac{1}{T}\enVert[1]{X(\hat{\beta}-\beta^*)}^2-\frac{2}{T}\epsilon'X(\hat{\beta}-\beta^*)+2\lambda_T\enVert[1]{\hat{\beta}}_{\ell_1}
\leq 
\frac{1}{T}\enVert{\epsilon}^2+2\lambda_T\enVert[1]{\beta^*}_{\ell_1}
\end{align*}
Or, equivalently
\begin{align}
\frac{1}{T}\enVert[1]{X(\hat{\beta}-\beta^*)}^2
\leq 
\frac{2}{T}\epsilon'X(\hat{\beta}-\beta^*)+2\lambda_T\left(\enVert[1]{\beta^*}_{\ell_1}-\enVert[1]{\hat{\beta}}_{\ell_1}\right)\label{start}
\end{align}
To bound $\frac{1}{T}\enVert[1]{X(\hat{\beta}-\beta^*)}^2$ one must bound $\frac{2}{T}\epsilon'X(\hat{\beta}-\beta^*)$. Note that on the set $\mathcal{B}_T$ defined in (\ref{BT}) one has
\begin{align*}
\frac{2}{T}\epsilon'X(\hat{\beta}-\beta^*)
\leq
2\enVert[2]{\frac{1}{T}\epsilon'X}_{\ell_\infty}\enVert[0]{\hat{\beta}-\beta^*}_{\ell_1}
\leq
\lambda_T\enVert[1]{\hat{\beta}-\beta^*}_{\ell_1}
\end{align*}
Putting things together, on $\mathcal{B}_T$,
\begin{align*}
\frac{1}{T}\enVert[1]{X(\hat{\beta}-\beta^*)}^2
\leq
\lambda_T\enVert[1]{\hat{\beta}-\beta^*}_{\ell_1}+2\lambda_T\left(\enVert[1]{\beta^*}_{\ell_1}-\enVert[1]{\hat{\beta}}_{\ell_1}\right)
\end{align*}
Adding $\lambda_T\enVert[1]{\hat{\beta}-\beta^*}_{\ell_1}$ yields
\begin{align}
\frac{1}{T}\enVert[1]{X(\hat{\beta}-\beta^*)}^2+\lambda_T\enVert[1]{\hat{\beta}-\beta^*}_{\ell_1}
\leq
2\lambda_T\left(\enVert[1]{\hat{\beta}-\beta^*}_{\ell_1}+\enVert[1]{\beta^*}_{\ell_1}-\enVert[1]{\hat{\beta}}_{\ell_1}\right)\label{A1}
\end{align}
which is inequality (\ref{IQ1}). To obtain inequality (\ref{IQ2}) notice that 
\begin{align*}
\enVert[1]{\hat{\beta}-\beta^*}_{\ell_1}+\enVert[1]{\beta^*}_{\ell_1}-\enVert[1]{\hat{\beta}}_{\ell_1}=\enVert[1]{\hat{\beta}_{J}-\beta^*_{J}}_{\ell_1}+\enVert[1]{\beta^*_{J}}_{\ell_1}-\enVert[1]{\hat{\beta}_{J}}_{\ell_1}
\end{align*}
In addition,  
\begin{align*}
\enVert[1]{\hat{\beta}_{J}-\beta^*_{J}}_{\ell_1}+\enVert[1]{\beta^*_{J}}_{\ell_1}-\enVert[1]{\hat{\beta}_{J}}_{\ell_1}
\leq 
2\enVert[1]{\hat{\beta}_{J}-\beta^*_{J}}_{\ell_1}
\end{align*}
by continuity of the norm. Furthermore,
\begin{align*}
\enVert[1]{\hat{\beta}_{J}-\beta^*_{J}}_{\ell_1}+\enVert[1]{\beta^*_{J}}_{\ell_1}-\enVert[1]{\hat{\beta}_{J}}_{\ell_1}\leq 2\enVert[1]{\beta^*_{J}}_{\ell_1}
\end{align*}
by subadditivity of the norm. Using the above two estimates in (\ref{A1}) yields inequality (\ref{IQ2}).
Next notice that (\ref{IQ2}) gives
\begin{align*}
\lambda_T\enVert[1]{\hat{\beta}-\beta^*}_{\ell_1}\leq 4\lambda_T \enVert[1]{\hat{\beta}_{J}-\beta^*_{J}}_{\ell_1}
\end{align*}
which is equivalent to
\begin{align*}
\enVert[1]{\hat{\beta}_{{J}^c}-\beta^*_{{J}^c}}_{\ell_1}\leq 3\enVert[1]{\hat{\beta}_{J}-\beta^*_{J}}_{\ell_1}
\end{align*}
and establishes (\ref{IQ3}). 
The lower bound on the probability with which (\ref{IQ1})-(\ref{IQ3}) hold follows from the fact that $P(\mathcal{B}_T)\geq 1-2(k^2p)^{1-\ln(1+T)}-2(1+T)^{-1/A}$ by Lemma \ref{Bound}. 
\end{proof}

The following lemma shows that in order to verify the restricted eigenvalue condition for a matrix it suffices that this matrix is close (in terms of maximum entrywise distance) to a matrix which does satisfy the restricted eigenvalue condition.

\begin{lemma}\label{vdGB}
Let $A$ and $B$ be two positive semi-definite $n\times n$ matrices and assume that $A$ satisfies the restricted eigenvalue condition RE($s$) for some $\kappa_A$. Then, for $\delta=\max_{1\leq i,j\leq n}\envert[0]{A_{i,j}-B_{i,j}}$, one also has $\kappa_B^2\geq \kappa_A^2-16s\delta$.
\end{lemma}
 
\begin{proof}
The proof is similar to Lemma 10.1 in \cite{vdGB09}. For any (non-zero) $n\times 1$ vector $v$ such that $\enVert[0]{v_{J^c}}_{\ell_1}\leq 3\enVert[0]{v_J}_{\ell_1}$ one has
\begin{align*}
v'Av-v'Bv
&\leq
\envert[0]{v'Av-v'Bv}
=
\envert[0]{v'(A-B)v}
\leq 
\enVert[0]{v}_{\ell_1}\enVert[0]{(A-B)v}_{\ell_\infty}
\leq
\delta\enVert[0]{v}_{\ell_1}^2\\
&\leq 
\delta 16\enVert[0]{v_J}_{\ell_1}^2
\leq
\delta 16s\enVert[0]{v_J}^2
\end{align*} 
Hence, rearranging the above, yields
\begin{align*}
v'Bv
\geq
v'Av - 16s\delta\enVert[0]{v_J}^2 
\end{align*}
or equivalently,
\begin{align*}
\frac{v'Bv}{v_J'v_J}
\geq
\frac{v'Av}{v_J'v_J}-16s\delta 
\geq
\kappa_A^2-16s\delta 
\end{align*}
Minimizing the left hand side over $\cbr[0]{v\in \mathbb{R}^n\setminus \{0\}: \enVert[0]{v_{J^c}}_{\ell_1}\leq 3\enVert[0]{v_J}_{\ell_1}}$ yields the claim.
\end{proof}

\begin{lemma}\label{Wain}
Let $V$ be an $n\times 1$ vector with $V\sim N(0,Q)$. Then, for any $\epsilon,M>0$, $P\del[2]{\envert[1]{\enVert[0]{V}^2-E\enVert[0]{V}^2}>\epsilon}\leq 2\exp\del[1]{\frac{-\epsilon^2}{8n\enVert{Q}_{\ell_\infty}^2M^2}}+n\exp\del[0]{-M^2/2}$ 
\end{lemma}

\begin{proof}
The statement of the lemma only depends on the distribution of $V$ and so we may equivalently consider $\sqrt{Q}\tilde{V}$ with $\tilde{V}\sim N(0,I)$ where $\sqrt{Q}$ is the matrix square root of $Q$. Hence, 
\begin{align}
&P\del[2]{\envert[1]{\enVert[1]{\sqrt{Q}\tilde{V}}^2-E\enVert[1]{\sqrt{Q}\tilde{V}}^2}>\epsilon}\notag\\
&\leq
P\del[2]{\envert[1]{\enVert[1]{\sqrt{Q}\tilde{V}}^2-E\enVert[1]{\sqrt{Q}\tilde{V}}^2}>\epsilon,\ \enVert[0]{\tilde{V}}_{\ell_\infty}\leq M}+P(\enVert[0]{\tilde{V}}_{\ell_\infty}>M)\label{Lip}
\end{align}
To get an estimate on the first probability we show that on the set $\cbr[0]{\enVert[0]{x}_{\ell_\infty}\leq M}$ the function $f(x)=\enVert[0]{\sqrt{Q}x}^2=\enVert[0]{\sqrt{Q}(-M\vee x\wedge M)}^2$ (the minimum and maximum are understood entrywise in the vector $x$) is Lipschitz continuous. Note that with $g(x)=(-M\vee x\wedge M)$ we can write $f(x)=f(g(x))$ on $\cbr[0]{\enVert[0]{x}_{\ell_\infty}\leq M}$. To obtain a bound on the Lipschitz constant note that by the mean value theorem, on $\cbr[0]{\enVert[0]{x}_{\ell_\infty}\leq M}$,
\begin{align*}
\envert{f(x)-f(y)}
&=
\envert{f(g(x))-f(g(y))}
=
\envert{f'(c)'(g(x)-g(y))}\\
&\leq
\enVert{f'(c)}_{\ell_\infty}\enVert{g(x)-g(y)}_{\ell_1}
\leq
\enVert{f'(c)}_{\ell_\infty}\enVert{x-y}_{\ell_1}
\leq
\enVert{f'(c)}_{\ell_\infty}\sqrt{n}\enVert{x-y}
\end{align*}
for a point $c$ on the line segment joining $g(x)$ and $g(y)$. Since $c=\mu g(x)+(1-\mu)g(y)$ for some $0<\mu<1$ one has $\enVert{c}_{\ell_\infty}\leq \mu \enVert{g(x)}_{\ell_\infty}+(1-\mu)\enVert{g(y)}_{\ell_\infty}\leq M$ and so
\begin{align*}
\enVert{f'(c)}_{\ell_\infty}\sqrt{n}
=\enVert{2Qc}_{\ell_\infty}\sqrt{n}
\leq
2\sqrt{n}\enVert{Q}_{\ell_\infty}\enVert{c}_{\ell_\infty}
\leq
2\sqrt{n}\enVert{Q}_{\ell_\infty}M
\end{align*}
Hence, $f(x)$ is Lipschitz with Lipschitz constant bounded by $2\sqrt{n}\enVert{Q}_{\ell_\infty}M$. The Borell-Cirelson-Sudakov inequality (see e.g. \cite{massart07}, Theorem 3.4) then yields that the first probability in (\ref{Lip}) can be be bounded by $2\exp\del[1]{\frac{-\epsilon^2}{2(2\sqrt{n}\enVert{Q}_{\ell_\infty}M)^2}}$. Regarding the second probability in (\ref{Lip}) note that by the union bound and standard tail probabilities for gaussian variables (see e.g. \cite{bill99}, page 263) one has $P(\enVert[0]{\tilde{V}}_{\ell_\infty}>M)\leq ne^{-M^2/2}$. This yields the lemma.
\end{proof}

\begin{lemma}\label{CovBound}
Let Assumption \ref{Ass1} be satisfied. Then, for any $t,M>0$, one has\\ $P\del[2]{\max_{1\leq i,j\leq kp}\envert[0]{\Psi_{T,{i,j}}-\Gamma_{i,j}}>t}
\leq
2k^2p^2\del[2]{2\exp\del[1]{\frac{-t^2T}{8\enVert{Q}_{\ell_\infty}^2M^2}}+T\exp\del[0]{-M^2/2}}$ where $\enVert[0]{Q}_{\infty}\leq 2\enVert[0]{\Gamma}\sum_{i=0}^T\enVert[0]{F^{i}}.
$
\end{lemma}

\begin{proof}
For any $t>0$, it follows from a union bound
\begin{align}
P\del[2]{\max_{1\leq i,j\leq kp}\envert[0]{\Psi_{T,{i,j}}-\Gamma_{i,j}}>t}
\leq 
k^2p^2\max_{1\leq i,j \leq kp}P\del[2]{\envert[0]{\Psi_{T,{i,j}}-\Gamma_{i,j}}>t}\label{dis}
\end{align}
Hence, it suffices to bound $P\del[2]{\envert[0]{\Psi_{T,{i,j}}-\Gamma_{i,j}}>t}$ appropriately for all $1\leq i,j\leq kp$. To this end note that by the stationarity of $y_t$
\begin{align*}
P\del[1]{\envert[0]{\Psi_{T,{i,j}}-\Gamma_{i,j}}>t}
=
P\del[1]{\envert[0]{(X'X)_{i,j}-E(X'X)_{i,j}}>tT}
\end{align*}
This further implies that it is enough to bound $P\del[1]{\envert[0]{v'X'Xv-E(v'X'Xv)}>tT}$ for any $t>0$ and column vector $v$ with $\enVert[0]{v}=1$ \footnote{Letting $v$ run over the standard basis vectors of $\mathbb{R}^{kp}$ yields the result for all diagonal elements. Choosing $v$ to contain only zeros except for $1/\sqrt{2}$ in the $i$th and $j$th position and thereafter only zeros except for $1/\sqrt{2}$ in the $i$th position and $-1/\sqrt{2}$ in the $j$th position some elementary calculations (available upon request) show that $P\del[1]{\envert[0]{(X'X)_{i,j}-E(X'X)_{i,j}}>tT}
$ for $i\neq j$ is bounded by two times $P\del[1]{\envert[0]{(X'X)_{i,i}-E(X'X)_{i,i}}>tT}$. The maximum on the right hand side in (\ref{dis}) is obtained for the off-diagonal elements. This explains the first $2$ on the right hand side in (\ref{display}).}. But for $U=Xv$ this probability equals $P\del[2]{\envert[1]{\enVert[0]{U}^2-E\enVert[0]{U}^2}>tT}$. $U$ is a linear transformation of a multivariate gaussian and hence gaussian itself with mean zero and covariance $Q:=E(UU')$. Hence, it follows from Lemma \ref{Wain}
\begin{align}
P\del[2]{\max_{1\leq i,j\leq kp}\envert[0]{\Psi_{T,{i,j}}-\Gamma_{i,j}}>t}
\leq 
2k^2p^2\del[2]{2\exp\del[1]{\frac{-t^2T}{8\enVert{Q}_{\ell_\infty}^2M^2}}+T\exp\del[0]{-M^2/2}}\label{display}
\end{align}

It remains to upper bound $\enVert{Q}_{\ell_\infty}=\max_{1\leq t\leq T}\sum_{s=1}^{T}|Q_{t,s}|$. For any pair of $1\leq s,t\leq T$ letting $\Gamma_{t-s}=E(Z_tZ_s')$ (clearly $\Gamma_0=\Gamma$) and writing $y_t$ in its companion form (as an VAR(1)) with companion matrix $F$: \footnote{Recall that $Z_{t}'=(y_{t-1}',...,y_{t-p}')$ is the $t$th row of $X$.}
\begin{align*}
|Q_{t,s}|
=
\envert{E(Z_t'vv'Z_s)}
=
\envert{E(v'Z_sZ_t'v)}
=
\envert{v'\Gamma_{s-t} v}
=
\begin{cases}
\envert[0]{v'F^{s-t}\Gamma v}&\text{ for }s\geq t\\
\envert[0]{v'\Gamma(F^{t-s})'v}=\envert[0]{v'F^{t-s}\Gamma v}&\text{ for }s<t
\end{cases}
\end{align*}
Hence, $|Q_{t,s}|=\envert[0]{v'F^{|t-s|}\Gamma v}$ for any pair of $1\leq s,t\leq T$. By the Cauchy-Schwarz inequality 
\begin{align*}
\envert[0]{v'F^{|t-s|}\Gamma v}
\leq
\enVert[0]{v'F^{|t-s|}}\enVert[0]{\Gamma v}
\leq
\enVert[0]{F^{|t-s|}}\enVert[0]{\Gamma}
\end{align*}
Putting things together yields (uniformly over $\cbr[0]{v:\enVert[0]{v}=1}$)
\begin{align*}
\enVert{Q}_{\ell_\infty}
=
\max_{1\leq t\leq T}\sum_{s=1}^{T}\enVert[0]{F^{|t-s|}}\enVert[0]{\Gamma}
\leq
2\enVert[0]{\Gamma}\sum_{i=0}^T\enVert[0]{F^{i}}
\end{align*}

\end{proof}

\begin{lemma}\label{REbound}
Let Assumption \ref{Ass1} be satisfied. Then, on 
\begin{align}
\mathcal{C}_T=\cbr[3]{\max_{1\leq i,j\leq T}\envert[0]{\Psi_{i,j}-\Gamma_{i,j}}\leq \frac{(1-q)\kappa_{\Gamma}^2(s)}{16s}}\label{CT}
\end{align}
one has for any $0<q<1$ and $s\in\cbr[0]{1,...,kp}$ that
\begin{itemize}
\item[i)] $\kappa_{\Psi_T}^2(s)\geq q\kappa_{\Gamma}^2(s)$ 
\item[ii)] $\phi_{\min}(\Psi_{J,J})\geq q \phi_{\min}(\Gamma_{J,J})$  
\end{itemize}
Finally, $P(\mathcal{C}_T)\geq 1-4k^2p^2\exp\del[2]{\frac{-\zeta T}{s^2\log(T)(\log(k^2p^2)+1)}}-2(k^2p^2)^{1-\log(T)}=1-\pi_q(s)$ for $\zeta=\frac{(1-q)^2\kappa_{\Gamma}^4}{4\cdot 16^3(\enVert[0]{\Gamma}\sum_{i=0}^T\enVert[0]{F^{i}})^2}$.
\end{lemma}

\begin{proof}
Define $\kappa_{\Gamma}(s)=\kappa_{\Gamma}$ and $\kappa_{\Psi_T}^2(s)=\kappa_{\Psi_T}^2$. By Lemma \ref{vdGB} one has that $\kappa_{\Psi_T}^2\geq q\kappa_{\Gamma}^2$ if $\max_{1\leq i,j\leq T}\envert[0]{\Psi_{i,j}-\Gamma_{i,j}}\leq \frac{(1-q)\kappa_{\Gamma}^2}{16s}$. Furthermore, an argument similar to the one in Lemma \ref{vdGB} reveals that $\phi_{\min}(\Psi_{J,J})\geq q\phi_{\min}(\Gamma_{J,J})$ if the maximal entry of $\envert[0]{\Psi_{J,J}-\Gamma_{J,J}}$ is less than $\frac{1-q}{s}\phi_{\min}(\Gamma_{J,J})$. But note that (in the display below, the maximum is to be understood entrywisely)
\begin{align*}
\mathcal{C}_T\subseteq \cbr[3]{\max\envert[0]{\Psi_{J,J}-\Gamma_{J,J}}\leq\frac{1-q}{s}\phi_{\min}(\Gamma_{J,J})}
\end{align*}
such that  $\phi_{\min}(\Psi_{J,J})\geq q\phi_{\min}(\Gamma_{J,J})$ on $\mathcal{C}_T$. It remains to lower bound the measure of $\mathcal{C}_T$. Using $M^2=2\log(k^2p^2)\log(T)+2\log(T)$ in Lemma \ref{CovBound} yields
\begin{align}
&P\del[2]{\max_{1\leq i,j\leq kp}\envert[0]{\Psi_{T,{i,j}}-\Gamma_{i,j}}>\frac{(1-q)\kappa_{\Gamma}^2}{16s}}\notag\\
\leq
&2k^2p^22\exp\del[2]{\frac{-(1-q)^2\kappa_{\Gamma}^4T}{16^2s^28\enVert{Q}_{\ell_\infty}^2(2\log(k^2p^2)\log(T)+2\log(T))}}+2(k^2p^2)^{1-\log(T)}\notag\\
=
&4k^2p^2\exp\del[2]{\frac{-\zeta T}{s^2\log(T)(\log(k^2p^2)+1)}}+2(k^2p^2)^{1-\log(T)}:=\pi_q(s)\label{lemmabound}
\end{align}
Inserting the upper bound on $\enVert{Q}_{\ell_\infty}$ from Lemma \ref{CovBound} yields the lemma upon taking the complement.
\end{proof}

\begin{proof}[Proof of Theorem \ref{Thm3}]
As the results are equation by equation we shall focus on equation $i$ here but omit the subscript $i$ for brevity when no confusion arises. We will work on $\mathcal{B}_T\cap \mathcal{C}_T$ as defined in (\ref{BT}) and (\ref{CT}). By (\ref{IQ2}), Jensen's inequality and the restricted eigenvalue condition (which is applicable due to (\ref{IQ3})) 
\begin{align*}
\frac{1}{T}\enVert[1]{X(\hat{\beta}-\beta^*)}^2
\leq
4\lambda_T\enVert[1]{\hat{\beta}_{J}-\beta^*_{J}}_{\ell_1}
\leq 
4\lambda_T\sqrt{s}\enVert[1]{\hat{\beta}_{J}-\beta^*_{J}}
\leq
4\lambda_T\sqrt{s}\frac{\enVert[1]{X(\hat{\beta}-\beta^*)}}{\kappa_{\Psi_T}\sqrt{T}}
\end{align*}
Rearranging and using $\kappa_{\Psi_T}^2\geq q\kappa^2$, yields (\ref{REIQ1}). 
To establish (\ref{REIQ2}) use (\ref{IQ3}), Jensen's inequality, $\kappa_{\Psi_T}^2\geq q\kappa^2$, and (\ref{REIQ1}):
\begin{align*}
\enVert[1]{\hat{\beta}-\beta^*}_{\ell_1}
\leq
4\enVert[1]{\hat{\beta}_{J}-\beta^*_{J}}_{\ell_1}
\leq
4\sqrt{s}\enVert[1]{\hat{\beta}_{J}-\beta^*_{J}}
\leq
4\sqrt{s}\frac{\enVert[1]{X(\hat{\beta}-\beta^*)}}{\kappa_{\Psi_T}\sqrt{T}}
\leq
\frac{16}{q\kappa^2}s\lambda_T
\end{align*}
Regarding retaining all relevant variables, let $\beta_{\min}>\frac{16}{q\kappa_{i}^2}s_{i}\lambda_T$. If there exists a $j\in J$ such that $\hat{\beta}_j=0$, then $\enVert[0]{\hat{\beta}-\beta^*}_{\ell_1}\geq \beta_{\min}> \frac{16}{q\kappa_{i}^2}s_{i}\lambda_T$ which contradicts (\ref{REIQ2}). Hence, $\hat{\beta}_j\neq 0$ for all $j\in J$. 

Combining Lemmas \ref{Bound} and \ref{REbound}, $\mathcal{B}_T\cap \mathcal{C}_T$ is seen to have at least the stated probability. Regarding the last assertion,
\begin{align}
\mathcal{C}_T(\bar{s}):=\cbr[3]{\max_{1\leq i,j\leq T}\envert[0]{\Psi_{i,j}-\Gamma_{i,j}}\leq \frac{(1-q)\kappa_{\Gamma}^2(\bar{s})}{16\bar{s}}}\subseteq \cbr[3]{\max_{1\leq i,j\leq T}\envert[0]{\Psi_{i,j}-\Gamma_{i,j}}\leq \frac{(1-q)\kappa^2_{\Gamma}(s_i)}{16s_i}}\label{incl}
\end{align}
for all $i=1,...,k$. On $\mathcal{C}_T(\bar{s})$ it follows from Lemma \ref{REbound} that $\kappa_{\Psi_T}^2(s_i)\geq q\kappa_{\Gamma}^2(s_i)$ for all $i=1,...,k$ which is exactly what we used in the arguments above. Hence, (\ref{REIQ1}) and (\ref{REIQ2}) are a valid and no variables are excluded for all $i=1,...,k$ on $\mathcal{B}_T\cap\mathcal{C}_T(\bar{s})$ which has probability at least $1-2(k^2p)^{1-\ln(1+T)}-2(1+T)^{-1/A}-\pi_q(\bar{s})$ by Lemmas \ref{Bound} and \ref{REbound}. 
\end{proof}

\begin{proof}[Proof of Lemma \ref{OLSO}]
As the results are equation by equation we shall focus on equation $i$ here but omit the subscript $i$ for brevity. Let $X_{J}$ denote the matrix consisting of the columns of $X$ indexed by $J$. Then for $1\leq i_h,j_h\leq k$ and $1\leq l_h\leq p$ on
\begin{align*}
\tilde{\mathcal{B}}_T=\left\{\max_{1\leq h\leq s}\envert{\frac{1}{T}\sum_{t=1}^Ty_{t-l_h,i_h}\epsilon_{t,j_h}}<\frac{\tilde{\lambda}_T}{2}\right\}
\end{align*}
one has, regarding $\left(\frac{1}{T}X_{J}'X_{J}\right)^{-1}$ as a bounded linear operator from $\ell_2(\mathbb{R}^s)$ to $\ell_2(\mathbb{R}^s)$ with induced operator norm given by $\phi_{\max}\left((\frac{1}{T}X_{J}'X_{J})^{-1}\right)=1/\phi_{\min}(\frac{1}{T}X_{J}'X_{J})$,
\begin{align*}
\enVert[1]{\hat{\beta}_{OLS}-\beta^*_J}_{\ell_1}
&\leq
\sqrt{s}\enVert{\left(\frac{1}{T}X_{J}'X_{J}\right)^{-1}\frac{1}{T}X_J'\epsilon}_{\ell_2}\\
&\leq 
\sqrt{s}\enVert{\left(\frac{1}{T}X_{J}'X_{J}\right)^{-1}}_{\ell_2}\enVert{\frac{1}{T}X_{J}'\epsilon}_{\ell_2}\\
&\leq
s\enVert{\left(\frac{1}{T}X_{J}'X_{J}\right)^{-1}}_{\ell_2}\enVert{\frac{1}{T}X_{J}'\epsilon}_{\ell_\infty}\\
&\leq
\frac{s}{2\phi_{\min}(\Psi_{J,J})}\tilde{\lambda}_T
\end{align*} 
Hence, it follows by Lemma \ref{REbound} that on $\mathcal{\tilde{B}}_T\cap\mathcal{C}_T$ 
\begin{align*}
\enVert[1]{\hat{\beta}_{OLS}-\beta^*_J}_{\ell_1}
&\leq
\frac{s}{2q\phi_{\min}(\Gamma_{J,J})}\tilde{\lambda}_T
\end{align*}
That the probability of $\tilde{\mathcal{B}}_T$ must be at least $1-2s^{1-\ln(1+T)}-2(1+T)^{-1/A}$ follows from a slight modification of Lemmas \ref{Orlicz} and \ref{Bound} using that one only has to bound terms associated with relevant variables. Hence, $\mathcal{\tilde{B}}_T\cap\mathcal{C}_T$ has at least the stated probability by combination with Lemma \ref{REbound}.
\end{proof}

\begin{proof}[Proof of Corollary \ref{Corthm3}]
Sum (\ref{REIQ2}) over $i=1,...,k$ which are all valid simultaneously with probability at least $1-2(k^2p)^{1-\ln(1+T)}-2(1+T)^{-1/A}-\pi_q(\bar{s})$ by Theorem \ref{Thm3}.
\end{proof}

\begin{lemma}\label{CTasym}
Assume that $k,p\in O(\exp(T^a))$ and $s\in O(T^b)$ for $2a+2b<1$. If $\kappa^2>c$ for some $c>0$ and $\sup_T\enVert[0]{\Gamma}\sum_{i=0}^T\enVert[0]{F^{i}}<\infty$. Then, $P(\mathcal{C}_T)\rightarrow 1$.
\end{lemma}

\begin{proof}
By Lemma \ref{REbound}, $P(\mathcal{C}_T)\geq 1-4k^2p^2\exp\del[2]{\frac{-\zeta T}{s^2\log(T)(\log(k^2p^2)+1)}}-2(k^2p^2)^{1-\log(T)}=1-\pi_q(s)$ for $\zeta=\frac{(1-q)^2\kappa_{\Gamma}^4}{4\cdot 16^3(\enVert[0]{\Gamma}\sum_{i=0}^T\enVert[0]{F^{i}})^2}$. It remains to be shown that $4k^2p^2\exp\del[2]{\frac{-\zeta T}{s^2\log(T)(\log(k^2p^2)+1)}}\rightarrow 0$. Noting that \begin{align*}
4k^2p^2\exp\del[2]{\frac{-\zeta T}{s^2\log(T)(\log(k^2p^2)+1)}}
=
4(k^2p^2)^{1-{\frac{-\zeta T}{s^2\log(T)\log(k^2p^2)(\log(k^2p^2)+1)}}}
\end{align*}
it suffices to show that $s^2\log(T)\log(k^2p^2)(\log(k^2p^2)+1)\in o(T)$. Now,
\begin{align*}
s^2\log(T)\log(k^2p^2)(\log(k^2p^2)+1)\in O(T^{2b}\log(T)T^a(T^a+1))\subseteq O(T^{2a+2b}\log(T))\subseteq o(T)
\end{align*}
since $2a+2b<1$.
\end{proof}

\begin{proof}[Proof of Theorem \ref{LASSOAsym}]
As the results are equation by equation we shall focus on equation $i$ here but omit the subscript $i$ for brevity. Observe that $s\lambda_T\rightarrow 0$ implies $\lambda_T<1$ from a certain step and onwards and so $s\lambda_T^2\rightarrow 0$ . Hence, i) and ii) follow from (\ref{REIQ1}) and (\ref{REIQ2}) of Theorem \ref{Thm3} if we show that $s\lambda_T\rightarrow 0$ and that the probability with which these hold tends to one.

First we show that $s\lambda_T\rightarrow 0$. $k,p\in O(e^{T^a})$ for some $a\geq 0$ implies that $1+k,1+p\in O(e^{T^a})$ since $e^{T^a}$ is bounded away from 0 (it tends to $\infty$ for $a>0$). Hence\footnote{By the definition of $\lambda_T=\sqrt{8\ln(1+T)^5\ln(1+k)^4\ln(1+p)^2\ln(k^2p)\sigma_T^4/T}$ one has that $\lambda_T\in O\left(\sqrt{\frac{\ln(1+T)^5T^{4a}T^{2a}(2T^a+T^a)}{T}}\right)= O\left(\sqrt{\frac{\ln(1+T)^5T^{7a}}{T}}\right)$.\label{footnotelambda}}, 
\begin{align}
s^2\lambda_T^2\in O\left(T^{2b}\frac{\ln(1+T)^5T^{4a}T^{2a}(2T^a+T^a)}{T}\right)
=
O\left(\ln(1+T)^5T^{7a+2b-1}\right) \subseteq o(1)\label{slambdaorder}
\end{align}
Noting that (\ref{REIQ1}) and (\ref{REIQ2}) hold on $\mathcal{B}_T\cap \mathcal{C}_T$ which is seen to have measure one asymptotically by Lemmas \ref{Bound} and \ref{CTasym} completes the proof of part i) and ii).
iii) follows from the fact that on $\mathcal{B}_T\cap \mathcal{C}_T$ no relevant variables are excluded if $\beta_{\min}>\frac{16}{qc^2}s\lambda_T$ (as argued in the proof of Theorem \ref{Thm3}). Since $\mathcal{B}_T\cap \mathcal{C}_T$ has probability 1 asymptotically this completes the proof.
\end{proof}

\begin{proof}[Proof of Theorem \ref{Lassoloc}]
As the results are equation by equation we shall focus on equation $i$ here but omit the subscript $i$ for brevity. The proof follows the same idea as that of Theorem \ref{LASSOAsym}. Hence, we show that the probability with which (\ref{REIQ2}) is valid tends to one and that the right hand side of the inequality tends to zero. To show the latter it suffices that $s\lambda_T\to 0$ which in turns is implied by $\sigma_T^4/T\to 0$. Since $k$ is fixed $\max_{1\leq i\leq k}\sigma_{i,\epsilon}$ is bounded and to control $\sigma_T^4$ it is enough to control $\max_{1\leq i\leq k}\sigma_{i,y}$ which may increase as $\rho=1-\frac{1}{T^\alpha}$ approaches one. Letting $w_t=(\epsilon_{t,1},...,\epsilon_{t,k},0,...,0)'$ be $kp\times 1$ one may write the VAR($p$) in its companion form $Z_t=FZ_{t-1}+w_t$ such that $\Gamma=E(Z_tZ_t')=\sum_{i=0}^\infty F^i\Omega {F^i}'$ where $\Omega=E(w_tw_t')$. Thus, $\sigma_{i,y}^2$ may be found as the $i$th diagonal element of $\Gamma$. Hence, letting $e_i$ be the $i$th basis vector in $\mathbb{R}^{kp}$ 
\begin{align}
\sigma_{i,y}^2
=
e_i'\Gamma e_i\leq \max_{\enVert{x}=1}x'\Gamma x
=
\enVert{\Gamma}\leq \sum_{i=0}^\infty\enVert[1]{F^i}\enVert{\Omega}\enVert[1]{{F^i}'}
=
\max_{1\leq j\leq k}\sigma_{j,\epsilon}^2\sum_{i=0}^\infty\enVert[1]{F^i}^2\label{Gamma}
\end{align}  
Next, since $F$ is assumed to have $kp$ distinct eigenvalues we can diagonalize it as $F=V^{-1}DV$ (see \cite{hornj90}) where the columns of $V$ are the linearly independent eigenvectors of $F$. Note that, for all $i\geq 1$, $F^i=V^{-1}D^iV$ and so $\enVert[0]{F^i}\leq \enVert[0]{V^{-1}}\enVert[0]{D^i}\enVert[0]{V}$. Clearly $\enVert[0]{D^i}=\rho^i$. Since we are considering a fixed model $\enVert[0]{V}$ is some finite number and because $V$ has linearly independent columns $V^{-1}$ is well defined with some finite $\ell_2$-norm. In total, $\enVert[0]{F^i}\leq C\rho^i$ for some $C>0$. Using this in (\ref{Gamma}) yields
\begin{align}
\sigma_{i,y}^2
\leq 
C\max_{1\leq j\leq k}\sigma_{j,\epsilon}^2\sum_{i=0}^\infty (\rho^2)^i
=
\frac{C}{1-\rho^2}\leq \frac{C}{1-\rho}
=
CT^\alpha \label{rho}
\end{align}
where the second estimate has used $\rho<1$ and the finite maximum $\max_{1\leq j\leq k}\sigma_{j,\epsilon}^2$ has been merged into the constant. Since the above display holds uniformly in $1\leq i\leq k$ we conclude $\sigma_T^4\in O(T^{2\alpha})$. Hence, since $\alpha<1/4$, one has $\sigma_T^4/T\to 0$.

To show that the probability with which (\ref{REIQ2}) is valid tends to one we only need to show that $\pi_q(s)\to 0$ for which we in turn show $4k^2p^2\exp\del[2]{\frac{-\zeta T}{s_i^2\log(T)(\log(k^2p^2)+1)}}\to 0$ with $\zeta=\frac{(1-q)^2\kappa_{i}^4}{4\cdot 16^3(\enVert[0]{\Gamma}\sum_{i=0}^T\enVert[0]{F^{i}})^2}$. $\enVert[0]{F^i}\leq C\rho^i$ for some $C>0$ yields
\begin{align*}
\sum_{i=0}^T\enVert[0]{F^{i}}
\leq 
\sum_{i=0}^\infty\enVert[0]{F^{i}} 
\leq
\frac{C}{1-\rho}=CT^\alpha
\in O(T^\alpha)
\end{align*}
Since (\ref{Gamma}) and (\ref{rho}) also yield $\enVert[0]{\Gamma}\in O(T^\alpha)$ we conclude $(\enVert[0]{\Gamma}\sum_{i=0}^T\enVert[0]{F^{i}})^2\in O(T^{4\alpha})$. Because $k,p$ and $s$ are fixed we conclude that $\pi_q(s)\to 0$ for $\alpha<1/4$.

\end{proof}


\begin{proof}[Proof of Theorem \ref{systemlassoasym}]
By Corollary \ref{Corthm3} it suffices to show that $k\bar{s}\lambda_T\rightarrow 0$ and that the probability with which (\ref{REIQcor2}) holds tends to one. Consider first $k\bar{s}\lambda_T$:
\begin{align}
k^2\bar{s}^2\lambda_T^2\in O\left(T^{2b}T^{2c}\frac{\ln(1+T)^5\ln(T)^4T^{2a}(\ln(T)+T^a)}{T}\right)
 \subseteq o(1)\label{slambdaorder2}
\end{align}

Since (\ref{REIQcor2}) holds on $\mathcal{B}_T\cap \mathcal{C}_T(\bar{s})$ where $\mathcal{C}_T(\bar{s})$ is defined in (\ref{incl}) and $P(\mathcal{B}_T)\rightarrow 1$ it remains to be shown that $P(\mathcal{C}_T(\bar{s}))\rightarrow 1$. This follows from an argument similar to the one in Lemma \ref{CTasym}. 
\end{proof}

\begin{lemma}\label{BoundY}
Let Assumption \ref{Ass1} be satisfied and define 
\begin{align}
\mathcal{D}_T=\left\{\max_{1\leq i,j\leq k}\max_{1\leq l,\tilde{l}\leq p}\envert[4]{\frac{1}{T}\sum_{t=1}^Ty_{t-l,i}y_{t-\tilde{l},j}}< K_T\right\}\label{DT}
\end{align}
for $K_T=\ln(1+k)^2\ln(1+p)^2\ln(T)\sigma_T^2$. Then,
\begin{align}
P(\mathcal{D}_T)\geq 1-2T^{-1/A}\label{DTprob}
\end{align}
for some constant $A>0$.
\end{lemma}

\begin{proof}
The proof is based on the same idea as in Lemma \ref{Orlicz} in its use of Orlicz norms. First bound
\begin{align*}
\enVert[3]{\max_{1\leq i,j\leq k}\max_{1\leq l, \tilde{l}\leq p} \envert[2]{\frac{1}{T}\sum_{t=1}^Ty_{t-l,i}y_{t-\tilde{l},j}}}_{\psi}
\end{align*}
where $\enVert[0]{\cdot}_{\psi}$ denotes the same Orlicz norm as in Lemma \ref{Orlicz}. To this end, notice that by the gaussianity of the $y_{t-l,i}$ for any $x>0$
\begin{align*}
&P\left(|y_{t-l,i}y_{t-\tilde{l},j}|\geq x\right)\leq P\left(|y_{t-l,i}|\geq\sqrt{x}\right)+P\left(|y_{t-\tilde{l},j}|\geq\sqrt{x}\right)\\
&\leq 2\exp(-x/\sigma^2_{i,y})+2\exp(-x/\sigma^2_{j,y})\leq 4\exp(-x/\sigma_T^2)
\end{align*}
Hence, $\cbr[0]{y_{t-l,i}y_{t-\tilde{l},j},\ t=1,...,T}$ has subexponential tails and it follows from Lemma 2.2.1 in \cite{vdVW96} that $\enVert[1]{y_{t-l,i}y_{t-\tilde{l},j}}_{\psi}\leq 10\sigma_T^2$. This implies that
\begin{align*}
\enVert[3]{\frac{1}{T}\sum_{t=1}^Ty_{t-l,i}y_{t-\tilde{l},j}}_{\psi}
\leq
\frac{1}{T}\sum_{t=1}^T\enVert{{y_{t-l,i}y_{t-\tilde{l},j}}}_\psi
\leq
10\sigma_T^2
\end{align*}
By Lemma \ref{OLemma} this implies that
\begin{align*}
\enVert[4]{\max_{1\leq i,j\leq k}\max_{1\leq l, \tilde{l}\leq p} \frac{1}{T}\sum_{t=1}^Ty_{t-l,i}y_{t-\tilde{l},j}}_{\psi} 
&\leq
K^4\ln(1+k)^2\ln(1+p)^210\sigma_T^2\\
&=
A\ln(1+k)^2\ln(1+p)^2\sigma_T^2 
\end{align*}
where $A=10K^4$. By the same trick as in Lemma \ref{Orlicz}
\begin{align*}
P\left(\max_{1\leq i,j\leq k}\max_{1\leq l, \tilde{l}\leq p} \envert[4]{\frac{1}{T}\sum_{t=1}^Ty_{t-l,i}y_{t-\tilde{l},j}}\geq K_T\right)
\leq
2\exp(-\ln(T)/A)=2T^{-1/A}
\end{align*}
\end{proof}

\begin{proof}[Proof of Theorem \ref{AdaLasso}]
As the results are equation by equation we shall focus on equation $i$ here but omit the subscript $i$ for brevity. Set $w=(1/|\hat{\beta}_1|,...,1/|\hat{\beta}_{kp}|)$ and $b=(\sgn(\beta^*_j)w_j)_{j\in J}$.
From \cite{zhou09} $\sgn(\hat{\beta})=\sgn(\beta^*)$ if and only if
\begin{align}
\envert{\Psi_{j,J}(\Psi_{J,J})^{-1}\left(\frac{X_{J}'\epsilon}{T}-\lambda_Tb\right)-\frac{X_{j}'\epsilon}{T}}\label{FOC1}
\leq
\lambda_Tw_{j}
\end{align}
for all $j\in J^c$ and
\begin{align}
\sgn\left(\beta^*_{J}+(\Psi_{J,J})^{-1}\left[\frac{X_{J}'\epsilon}{T}-\lambda_Tb\right]\right)
=
\sgn(\beta^*_{J})\label{FOC2}
\end{align}

We shall be working on $\mathcal{B}_T\cap\mathcal{C}_T\cap \mathcal{D}_T$ where each of the sets are defined in (\ref{BT}), (\ref{CT}), and (\ref{DT}), respectively. Consider (\ref{FOC1}) for a given $j\in {J}^c$. By the triangle inequality it suffices to show that
\begin{align}
\envert{\Psi_{j,J}(\Psi_{J,J})^{-1}\left(\frac{X_{J}'\epsilon}{T}-\lambda_Tb\right)}+\envert{\frac{X_{j}'\epsilon}{T}}\leq \lambda_Tw_{j}\label{FOC1Suff}
\end{align}
Bound the first term on the left hand side as follows:
\begin{align*}
\envert{\Psi_{j,J}(\Psi_{J,J})^{-1}\left(\frac{X_J'\epsilon}{T}-\lambda_Tb\right)}
&\leq
\enVert{\Psi_{j,J}(\Psi_{J,J})^{-1}}_{\ell_1}\enVert{\frac{X_J'\epsilon}{T}-\lambda_Tb}_{\ell_\infty}\\
&\leq
\sqrt{s}\enVert{\Psi_{j,J}(\Psi_{J,J})^{-1}}_{\ell_2}\left(\enVert[3]{\frac{X_{J}'\epsilon}{T}}_{\ell_\infty}+\enVert{\lambda_Tb}_{\ell_\infty}\right)
\end{align*}
Considering $(\Psi_{J,J})^{-1}$ as a bounded linear operator $\ell_2(\mathbb{R}^s)\rightarrow \ell_2(\mathbb{R}^s)$, the induced operator norm is given by $\phi_{\max}((\Psi_{J,J})^{-1})=1/\phi_{\min}(\Psi_{J,J})$ and so
\begin{align*}
\enVert[1]{\Psi_{j,J}(\Psi_{J,J})^{-1}}_{\ell_2}
\leq
\frac{\enVert[1]{\Psi_{j,J}}_{\ell_2}}{\phi_{\min}(\Psi_{J,J})}
\leq
\frac{\sqrt{s}\enVert[1]{\Psi_{j,J}}_{\ell_\infty}}{\phi_{\min}(\Psi_{J,J})}
\leq
\frac{\sqrt{s}K_T}{q\phi_{\min}(\Gamma_{J,J})}
\end{align*}
where the last estimate holds on $\mathcal{C}_T\cap \mathcal{D}_T$. By Lemma \ref{Bound} it follows that on $\mathcal{B}_T$
\begin{align}
\enVert[2]{\frac{X_{J}'\epsilon}{T}}_{\ell_\infty}\leq \frac{\lambda_T}{2}\label{AL1}
\end{align}
Next, since $\beta_{\min}\geq 2\enVert[0]{\hat{\beta}-\beta^*}_{\ell_1}$ one has for all $j\in J$,
\begin{align*}
|\hat{\beta}_j|
\geq 
|\beta_j^*|-|\hat{\beta}_j-\beta_j^*|
\geq
\beta_{\min}-\enVert[0]{\hat{\beta}_j-\beta_j^*}_{\ell_1}
\geq \beta_{\min}/2
\end{align*}
one gets
\begin{align}
&\enVert{\lambda_Tb}_{\ell_\infty}
=
\enVert{\lambda_Tw_{J}}_{\ell_\infty}
=
\lambda_T\max_{j\in J}\envert[3]{\frac{1}{\hat{\beta}_j}}
\leq
\frac{2\lambda_T}{\beta_{\min}}\label{AL2}.
\end{align}
Lastly, on $\mathcal{B}_T$,
\begin{align*}
\envert[3]{\frac{X_j'\epsilon}{T}}\leq \frac{\lambda_T}{2}
\end{align*}
for every $j\in {J}^c$. Hence, uniformly in $j\in {J}^c$,
\begin{align*}
\envert[3]{\Psi_{j,J}(\Psi_{J,J})^{-1}\left(\frac{X_{J}'\epsilon}{T}-\lambda_Tb\right)}+\envert[3]{\frac{X_j'\epsilon}{T}}
\leq
\frac{sK_T}{q\phi_{\min}(\Gamma_{J,J})}\left(\frac{\lambda_T}{2}+\frac{2\lambda_T}{\beta_{\min}}\right)+\frac{\lambda_T}{2}
\end{align*} 
Now bound the right hand side in (\ref{FOC1Suff}) from below. For every $j\in {J}^c$
\begin{align*}
|\lambda_Tw_{j}|=\lambda_T\frac{1}{|\hat{\beta}_j|}
\geq 
\lambda_T\frac{1}{\enVert[0]{\hat{\beta}-\beta^*}_{\ell_1}}
\end{align*}
This implies that (\ref{FOC1Suff}), and hence (\ref{FOC1}), is satisfied if 
\begin{align*}
\frac{sK_T}{q\phi_{\min}(\Gamma_{J,J})}\left(\frac{\lambda_T}{2}+\frac{2\lambda_T}{\beta_{\min}}\right)+\frac{\lambda_T}{2}
\leq
\lambda_T\frac{1}{\enVert[0]{\hat{\beta}-\beta^*}_{\ell_1}}
\end{align*}
or equivalently
\begin{align*}
\frac{sK_T}{q\phi_{\min}(\Gamma_{J,J})}\left(\frac{1}{2}+\frac{2}{\beta_{\min}}\right)\enVert[0]{\hat{\beta}-\beta^*}_{\ell_1}+\frac{\enVert[0]{\hat{\beta}-\beta^*}_{\ell_1}}{2}
\leq
1
\end{align*}
which is (\ref{AdaLasso1}). To verify (\ref{AdaLasso2}) it suffices to show that
\begin{align}
\enVert[4]{(\Psi_{J,J})^{-1}\left(\frac{X_{J}'\epsilon}{T}-\lambda_Tb\right)}_{\ell_\infty}
\leq
\beta_{\min}\label{FOC2Suff}
\end{align}
Considering $(\Psi_{J,J})^{-1}$ as a bounded linear operator $\ell_\infty(\mathbb{R}^s)\rightarrow\ell_\infty(\mathbb{R}^s)$ it follows that:
\begin{align*}
\enVert[3]{(\Psi_{J,J})^{-1}\left(\frac{X_{J}'\epsilon}{T}-\lambda_Tb\right)}_{\ell_\infty}
&\leq
\enVert[1]{(\Psi_{J,J})^{-1}}_{\ell_\infty}\enVert[3]{\frac{X_{J}'\epsilon}{T}-\lambda_Tb}_{\ell_\infty}\\
&\leq
\sqrt{s}\enVert[1]{(\Psi_{J,J})^{-1}}_{\ell_2}\left(\enVert[3]{\frac{X_{J}'\epsilon}{T}}_{\ell_\infty}+\enVert{\lambda_Tb}_{\ell_\infty}\right)\\
&\leq
\frac{\sqrt{s}}{q\phi_{\min}(\Gamma_{J,J})}\left(\frac{\lambda_T}{2}+\frac{2\lambda_T}{\beta_{\min}}\right)
\end{align*}
where the second estimate uses that $\enVert[1]{(\Psi_{J,J})^{-1}}_{\ell_\infty}\leq\sqrt{s}\enVert[1]{(\Psi_{J,J})^{-1}}_{\ell_2}$, c.f. \cite{hornj90} page 314, and the last estimate follows from (\ref{AL1}) and (\ref{AL2}) and the fact that we are working on $\mathcal{C}_T$. Inserting into (\ref{FOC2Suff}) completes the proof since $P(\mathcal{B}_T\cap \mathcal{C}_T\cap \mathcal{D}_T)$ has the desired lower bound by Lemmas \ref{Bound}, \ref{REbound}, and \ref{BoundY}. 
\end{proof}

\begin{proof}[Proof of Theorem \ref{AdaLASSOAsym}]
As the results are equation by equation we shall focus on equation $i$ here but omit the subscript $i$ for brevity. We continue to work on the set $\mathcal{B}_T\cap \mathcal{C}_T\cap \mathcal{D}_T$ from Theorem \ref{AdaLasso}. To show asymptotic sign consistency we verify that the conditions of Theorem \ref{AdaLasso} are valid asymptotically, i.e. $\beta_{\min}\geq 2\enVert[1]{\hat{\beta}-\beta^*}_{\ell_1}$ is valid asymptotically, (\ref{AdaLasso1}) and (\ref{AdaLasso2}) hold asymptotically, and that the probability of $\mathcal{B}_T\cap \mathcal{C}_T\cap \mathcal{D}_T$ tends to one. Now $\enVert[1]{\hat{\beta}-\beta^*}_{\ell_1}\leq \frac{16}{q\kappa^2}s\lambda_T$ by (\ref{REIQ2}) (which holds on $\mathcal{B}_T\cap \mathcal{C}_T$ as seen in the proof of Theorem \ref{Thm3}). Hence, since we have also seen that $\mathcal{B}_T\cap \mathcal{C}_T$ has probability one asymptotically, using that $\kappa$ is bounded away from zero, and $\beta_{\min}\in \Omega(\ln(T)a_T)$
\begin{align*}
\frac{\enVert[1]{\hat{\beta}_{i}-\beta_{i}^*}_{\ell_1}}{\beta_{\min}}
\in
O_p\del[3]{\frac{\ln(1+T)^{5/2}T^{7/2a+b-1/2}}{T^{2b}
T^{(15/2)a-1/2}\ln(T)^{2+5/2}}}
=
o_p(1)
\end{align*}
establishing that $\beta_{\min}\geq 2\enVert[1]{\hat{\beta}-\beta^*}_{\ell_1}$ with probability tending to one.

In order to verify the asymptotic validity of (\ref{AdaLasso1}) on $\mathcal{B}_T\cap \mathcal{C}_T\cap \mathcal{D}_T$ it suffices to show that (using $\enVert[1]{\hat{\beta}-\beta^*}_{\ell_1}\leq \frac{16}{q\kappa^2}s\lambda_T$)
\begin{align}
&\frac{s^2K_T\lambda_T}{\phi_{\min}(\Gamma_{J,J})\kappa^2}+\frac{s^2K_T\lambda_T}{\phi_{\min}(\Gamma_{J,J})\beta_{\min}\kappa^2}+\frac{s\lambda_T}{\kappa^2}
\rightarrow
0 \label{Aux1}
\end{align}
To this end, we show that each of the three terms tends to zero. Using that $\kappa^2$ and hence $\phi_{\min}(\Gamma_{J,J})$ are bounded away from 0  and $\sup_T\sigma_T<\infty$ one gets

\begin{align*}
\frac{s^2K_T\lambda_T}{\phi_{\min}(\Psi_{J,J})\kappa^2}
\in 
O\left(s^2K_T\lambda_T\right)
\subseteq
O\left(T^{2b}T^{(15/2)a-1/2}\ln(T)^{1+5/2}\right)=O(a_T)
\end{align*}
Because 
\begin{align*}
a_T^2=T^{15a+4b-1}\ln(T)^7\rightarrow 0
\end{align*}
it follows that the first term in (\ref{Aux1}) tends to zero. Since $\beta_{\min}\in \Omega(\ln(T)[a_T\vee b_T])$ and the second term equals the first term divided by $\beta_{\min}$ it follows that the second term also tends to zero. The third term tends to zero by (\ref{slambdaorder})
since $7a+2b\leq 15a+4b<1$. Next, it is shown that (\ref{AdaLasso2}) is valid asymptotically. To do so it suffices to show that
\begin{align}
\frac{\sqrt{s}\lambda_T}{\phi_{\min}(\Gamma_{J,J})\beta_{\min}}+\frac{\sqrt{s}\lambda_T}{\phi_{\min}(\Gamma_{J,J})\beta_{\min}^2}\rightarrow 0\label{Aux2}
\end{align}
We show that each of the two terms tends to zero. Note that since $\phi_{\min}(\Gamma_{J,J})$ is bounded away from 0 and $\beta_{\min}\in \Omega(\ln(T)a_T)$
\begin{align*}
\frac{s^{1/2}\lambda_T}{\phi_{\min}(\Psi_{J,J})\beta_{\min}}
\in 
O\left(\frac{s^{1/2}\lambda_T}{\ln(T)a_T}\right)
\subseteq 
O\left(\frac{T^{b/2}T^{(7/2)a-1/2}\ln(T)^{5/2}}{\ln(T)T^{2b}T^{(15/2)a-1/2}\ln(T)^{1+5/2}}\right)\subseteq o(1) 
\end{align*}
Regarding the second term in (\ref{Aux2}) it follows from $\beta_{\min}\in \Omega(\ln(T)b_T)$ that 
\begin{align*}
\frac{\sqrt{s}\lambda_T}{\phi_{\min}(\Gamma_{J,J})\beta_{\min}^2}
\in 
O\del[3]{\frac{T^{b/2}T^{(7/2)a-1/2}\ln(T)^{5/2}}{\ln(T)^2T^{b/2}T^{(7/2)a-1/2}\ln(T)^{5/2}}}\subseteq o(1)
\end{align*}

Finally, we see $\mathcal{B}_T\cap \mathcal{C}_T\cap\mathcal{D}_T$ has asymptotic measure one by Lemma \ref{CTasym} and (\ref{BTprob}) and (\ref{DTprob}).

\end{proof}

\begin{proof}[Proof of Theorem \ref{asymdist}]
As the results are equation by equation we shall focus on equation $i$ here but omit the subscript $i$ for brevity. The notation is as in the statement of Theorem \ref{AdaLASSOAsym}. Under the assumptions of that theorem, the adaptive LASSO is sign consistent. Hence, with probability tending to one, the estimates of the non-zero coefficients satisfy the first order condition
\begin{align*}
\frac{\partial L(\beta)}{\partial \beta_J}=-2X_J'(y-X\tilde{\beta})+2\lambda_Tg=0
\end{align*}
where $g=(\sgn(\beta^*_j)/|\hat{\beta}_j|)_{j\in J}$.
Using that $y=X_J\beta^*_J+\epsilon$ this is equivalent to
\begin{align*}
-2X_J'\left(\epsilon -X_J(\tilde{\beta}_J-\beta^*_J)-X_{J^c}\tilde{\beta}_{J^c}\right)+2\lambda_Tg=0
\end{align*}
and so, with probability tending to one, for any $s\times 1$ vector $\alpha$ with norm 1 one has
\begin{align}
&\sqrt{T}\alpha'(\tilde{\beta}_J-\beta^*_J)\notag\\
&=
\frac{1}{\sqrt{T}}\alpha'(\Psi_{J,J})^{-1}X_J'\epsilon-\sqrt{T}\alpha'(\Psi_{J,J})^{-1}\Psi_{J,J^c}\tilde{\beta}_{J^c}-\frac{\lambda_T}{\sqrt{T}}\alpha'(\Psi_{J,J})^{-1}g \label{OLSaux}
\end{align}
The first term on the right hand side in (\ref{OLSaux}) is recognized as $\sqrt{T}\alpha'(\hat{\beta}_{OLS}-\beta^*_J)$. Hence, to establish the theorem, it suffices to show that the second and the third term on the right hand side tend to zero in probability. Since $P(\tilde{\beta}_{J^c}=0)\rightarrow 1$ the second term vanishes in probability. Regarding the third term, notice that
\begin{align*}
\envert{\frac{\lambda_T}{\sqrt{T}}\alpha'(\Psi_{J,J})^{-1}g}
\leq
\frac{\lambda_T}{\sqrt{T}}\envert{\alpha'(\Psi_{J,J})^{-1}g}
\leq
\sqrt{\alpha'(\Psi_{J,J})^{-2}\alpha}\frac{\lambda_T}{\sqrt{T}}\sqrt{g'g}
\end{align*}
Now, on $\mathcal{C}_T$, which has probability tending to one, $\phi_{\min}(\Psi_{J,J})\geq q\phi_{\min}(\Gamma_{J,J})$ (by Lemma \ref{REbound}), so 
\begin{align*}
\alpha'\Psi_{J,J}^{-2}\alpha \leq \alpha'\alpha \phi_{\max}((\Psi_{J,J})^{-2})
=\alpha'\alpha/\phi_{\min}((\Psi_{J,J})^2)
\leq 
\alpha'\alpha/(q^2\phi_{\min}^2(\Gamma_{J,J}))
\leq 1/(q^2\tilde{c}^2)
\end{align*}
since $\alpha$ has norm one. Note that for all $j\in J$
\begin{align*}
|\hat{\beta}_j|
\geq 
|\beta_j^*|-|\hat{\beta}_j-\beta_j^*|
\geq
|\beta_{\min}|-\enVert[0]{\hat{\beta}-\beta^*}_{\ell_1}
\end{align*}
and so by subadditivity of $x\mapsto \sqrt{x}$
\begin{align*}
\frac{\lambda_T}{\sqrt{T}}\sqrt{g'g}
=
\frac{\lambda_T}{\sqrt{T}}\sqrt{\sum_{j\in J}\frac{1}{\hat{\beta}_j^2}}
\leq
\frac{\lambda_T}{\sqrt{T}}\frac{s}{|\beta_{\min}|-\enVert[0]{\hat{\beta}-\beta^*}_{\ell_1}}
=
\frac{s\lambda_T}{\sqrt{T}|\beta_{\min}|}\frac{1}{1-\enVert[0]{\hat{\beta}-\beta^*}_{\ell_1}/|\beta_{\min}|}
\end{align*}
Since $\kappa \geq \tilde{c}$ it follows from (\ref{REIQ2}) that $\enVert[0]{\hat{\beta}-\beta^*}_{\ell_1}\in O_p(s\lambda_T)$. But $s\lambda_T\in O(\ln(1+T)^{5/2}T^{{(7/2)}a+b-1/2})$ by (\ref{slambdaorder}). Hence, $\enVert[0]{\hat{\beta}-\beta^*}_{\ell_1}\in O_p(\ln(1+T)^{5/2}T^{{(7/2)}a+b-1/2})$ and so, since $\beta_{\min}\in\Omega(\ln(T)[a_T\vee b_T])\subseteq \Omega(\ln(T)a_T)=\Omega(T^{2b}T^{(15/2)a-1/2}\ln(T)^{2+5/2})$, 
\begin{align*}
\frac{\enVert[0]{\hat{\beta}-\beta^*}_{\ell_1}}{|\beta_{\min}|}
\in 
O_p\left(\frac{\ln(1+T)^{5/2}T^{{(7/2)}a+b-1/2}}{T^{2b}T^{(15/2)a-1/2}\ln(T)^{2+5/2}}\right)=O_p\left(\frac{\ln(1+T)^{5/2}}{\ln(T)^{2+5/2}}T^{-4a-b}\right)
\end{align*}
Since $\frac{\ln(1+T)^{5/2}}{\ln(T)^{2+5/2}}T^{-4a-b}\rightarrow 0$ it follows that $\frac{\enVert[0]{\hat{\beta}-\beta^*}_{\ell_1}}{\beta_{\min}}\in o_p(1)$. Also, $15a+4b<1$ is more than sufficient for (\ref{slambdaorder}) to yield that $s\lambda_T\rightarrow 0$ and $\beta_{\min}\in \Omega\left(\ln(T)([b_T\vee c_T]\right)\subseteq \Omega\left(\ln(T)T^{-1/4}\right)$ implies that $\sqrt{T}\beta_{\min}\rightarrow\infty$ and the theorem follows. 
\end{proof}

\begin{proof}[Proof of Corollary \ref{corrate}]
As the results are equation by equation we shall focus on equation $i$ here but omit the subscript $i$ for brevity. Let $\epsilon>0$ be given. Then, 
\begin{align*}
\cbr{\sup_{\alpha:\enVert[0]{\alpha}\leq 1}\sqrt{T}\envert[1]{\alpha'(\tilde{\beta}_J-\beta_{OLS})}<\epsilon}
&\subseteq 
\bigcap_{j\in J}\cbr{\sqrt{T}\envert[1]{\tilde{\beta}_j-\beta_{OLS,j}}<\epsilon}\\
&\subseteq
\cbr{\sqrt{T}\enVert[1]{(\tilde{\beta}_J-\beta_{OLS})}_{\ell_1}<s\epsilon}\\
&=\cbr{\frac{\sqrt{T}}{s}\enVert[1]{(\tilde{\beta}_J-\beta_{OLS})}_{\ell_1}<\epsilon}
\end{align*}
And so
\begin{align*}
P\left(\frac{\sqrt{T}}{s}\enVert[1]{(\tilde{\beta}_J-\beta_{OLS})}_{\ell_1}<\epsilon\right)
\geq
P\left(\sup_{\alpha:\enVert[0]{\alpha}\leq 1}\sqrt{T}\envert[1]{\alpha'(\tilde{\beta}_J-\beta_{OLS})}<\epsilon\right)
\rightarrow 1
\end{align*}
by Theorem \ref{asymdist}. Since $\epsilon>0$ was arbitrary, this proves that $\enVert[1]{(\tilde{\beta}_J-\beta_{OLS})}_{\ell_1}\in o_p\left(\frac{s}{\sqrt{T}}\right)$.
Hence, by the triangle inequality and Lemma \ref{OLSO}
\begin{align*}
\enVert[0]{\tilde{\beta}_J-\beta_{J}^*}_{\ell_1}
\leq
\enVert[0]{\tilde{\beta}_J-\beta_{OLS}}_{\ell_1}+\enVert[0]{\tilde{\beta}_{OLS}-\beta_{J}^*}_{\ell_1} \in O_p(\tilde{\lambda}_Ts)
\end{align*}
since $\tilde{\lambda}_Ts>\frac{s}{\sqrt{T}}$.
\end{proof}

\bibliographystyle{chicagoa}	
\bibliography{references}		

\begin{thebibliography}{}

\bibitem[\protect\citeauthoryear{Bai and Ng}{Bai and Ng}{2008}]{baing08}
Bai, J. and S.~Ng (2008).
\newblock Large dimensional factor analysis.
\newblock {\em Foundations and Trends in Econometrics\/}~{\em 3}, 89--163.


\bibitem[\protect\citeauthoryear{Belloni, Chen, Chernozhukov, and
  Hansen}{Belloni et~al.}{2012}]{belloni12sparse}
Belloni, A., D.~Chen, V.~Chernozhukov, and C.~Hansen (2012).
\newblock Sparse models and methods for optimal instruments with an application
  to eminent domain.
\newblock {\em Econometrica\/}~{\em 80\/}(6), 2369--2429.


\bibitem[\protect\citeauthoryear{Belloni and Chernozhukov}{Belloni and
  Chernozhukov}{2011}]{bellonic11}
Belloni, A. and V.~Chernozhukov (2011).
\newblock High dimensional sparse econometric models: An introduction.
\newblock {\em Inverse Problems and High-Dimensional Estimation\/}, 121--156.


\bibitem[\protect\citeauthoryear{Belloni, Chernozhukov, and Wang}{Belloni
  et~al.}{2011}]{bellonicw11}
Belloni, A., V.~Chernozhukov, and L.~Wang (2011).
\newblock Square-root lasso: pivotal recovery of sparse signals via conic
  programming.
\newblock {\em Biometrika\/}~{\em 98\/}(4), 791--806.


\bibitem[\protect\citeauthoryear{Bernanke, Boivin, and Eliasz}{Bernanke
  et~al.}{2005}]{bernankebe05}
Bernanke, B., J.~Boivin, and P.~Eliasz (2005).
\newblock Measuring the effects of monetary policy: a factor-augmented vector
  autoregressive (favar) approach.
\newblock {\em The Quarterly Journal of Economics\/}~{\em 120\/}(1), 387--422.


\bibitem[\protect\citeauthoryear{Bickel, Ritov, and Tsybakov}{Bickel
  et~al.}{2009}]{bickelrt09}
Bickel, P.~J., Y.~Ritov, and A.~B. Tsybakov (2009).
\newblock Simultaneous analysis of lasso and dantzig selector.
\newblock {\em The Annals of Statistics\/}~{\em 37\/}(4), 1705--1732.


\bibitem[\protect\citeauthoryear{Billingsley}{Billingsley}{1999}]{bill99}
Billingsley, P. (1999).
\newblock {\em Convergence of Probability Measures\/} (second ed.).
\newblock John Wiley \& Sons.


\bibitem[\protect\citeauthoryear{Breiman}{Breiman}{1996}]{breiman96}
Breiman, L. (1996).
\newblock Heuristics of instability and stabilization in model selection.
\newblock {\em The Annals of Statistics\/}~{\em 24\/}(6), 2350--2383.


\bibitem[\protect\citeauthoryear{B{\"u}hlmann and van~de Geer}{B{\"u}hlmann and
  van~de Geer}{2011}]{bühlmannvdg11}
B{\"u}hlmann, P. and S.~van~de Geer (2011).
\newblock {\em Statistics for High-Dimensional Data: Methods, Theory and
  Applications}.
\newblock Springer-Verlag, New York.


\bibitem[\protect\citeauthoryear{Callot and Kock}{Callot and
  Kock}{2014}]{kockc12oracle}
Callot, L. A.~F. and A.~B. Kock (2014).
\newblock Oracle efficient estimation and forecasting with the adaptive lasso
  and the adaptive group lasso in vector autoregressions.
\newblock In N.~Haldrup, M.~Meitz, and P.~Saikkonen (Eds.), {\em Essays in
  Nonlinear Time Series Econometrics}. Oxford University Press.

\bibitem[\protect\citeauthoryear{Cand\`{e}s and Tao}{Cand\`{e}s and
  Tao}{2007}]{candest07}
Cand\`{e}s, E. and T.~Tao (2007).
\newblock The dantzig selector: Statistical estimation when p is much larger
  than n.
\newblock {\em The Annals of Statistics\/}~{\em 35}, 2313--2351.


\bibitem[\protect\citeauthoryear{Caner and Knight}{Caner and
  Knight}{2013}]{canerk13}
Caner, M. and K.~Knight (2013).
\newblock An alternative to unit root tests: Bridge estimators differentiate
  between nonstationary versus stationary models and select optimal lag.
\newblock {\em Journal of Statistical Planning and Inference\/}~{\em 143},
  691--715.


\bibitem[\protect\citeauthoryear{Fan and Li}{Fan and Li}{2001}]{fanl01}
Fan, J. and R.~Li (2001).
\newblock Variable selection via nonconcave penalized likelihood and its oracle
  properties.
\newblock {\em Journal of the American Statistical Association\/}~{\em
  96\/}(456), 1348--1360.


\bibitem[\protect\citeauthoryear{Fan and Lv}{Fan and Lv}{2008}]{fanl08}
Fan, J. and J.~Lv (2008).
\newblock Sure independence screening for ultrahigh dimensional feature space.
\newblock {\em Journal of the Royal Statistical Society: Series B (Statistical
  Methodology)\/}~{\em 70}, 849--911.


\bibitem[\protect\citeauthoryear{Hoffmann-J{\o}rgensen}{Hoffmann-J{\o}rgensen}{1994}]{hoffmann94}
Hoffmann-J{\o}rgensen, J. (1994).
\newblock {\em Probability with a view towards statistics. 1}.
\newblock CRC Press.


\bibitem[\protect\citeauthoryear{Horn and Johnson}{Horn and
  Johnson}{1990}]{hornj90}
Horn, R. and C.~Johnson (1990).
\newblock {\em Matrix Analysis}.
\newblock Cambridge University Press.


\bibitem[\protect\citeauthoryear{Huang, Horowitz, and Ma}{Huang
  et~al.}{2008}]{huanghm08}
Huang, J., J.~L. Horowitz, and S.~Ma (2008).
\newblock Asymptotic properties of bridge estimators in sparse high-dimensional
  regression models.
\newblock {\em The Annals of Statistics\/}~{\em 36}, 587--613.


\bibitem[\protect\citeauthoryear{Kock}{Kock}{2012}]{kock12}
Kock, A.~B. (2012).
\newblock Consistent and conservative model selection in stationary and
  non-stationary autoregressions.
\newblock {\em Submitted\/}.


\bibitem[\protect\citeauthoryear{Leeb and P{\"o}tscher}{Leeb and
  P{\"o}tscher}{2005}]{leebp05}
Leeb, H. and B.~M. P{\"o}tscher (2005).
\newblock Model selection and inference: Facts and fiction.
\newblock {\em Econometric Theory\/}~{\em 21}, 21--59.


\bibitem[\protect\citeauthoryear{Liao and Phillips}{Liao and
  Phillips}{2013}]{liaop13}
Liao, Z. and P.~Phillips (2013).
\newblock Automated estimation of vector error correction models.


\bibitem[\protect\citeauthoryear{Ludvigson and Ng}{Ludvigson and
  Ng}{2009}]{ludvigson2009macro}
Ludvigson, S. and S.~Ng (2009).
\newblock Macro factors in bond risk premia.
\newblock {\em Review of Financial Studies\/}~{\em 22\/}(12), 5027--5067.


\bibitem[\protect\citeauthoryear{Massart}{Massart}{2007}]{massart07}
Massart, P. (2007).
\newblock Concentration inequalities and model selection.


\bibitem[\protect\citeauthoryear{Meinshausen and B{\"u}hlmann}{Meinshausen and
  B{\"u}hlmann}{2006}]{meinshausenb06}
Meinshausen, N. and P.~B{\"u}hlmann (2006).
\newblock High-dimensional graphs and variable selection with the lasso.
\newblock {\em The Annals of Statistics\/}~{\em 34}, 1436--1462.


\bibitem[\protect\citeauthoryear{Nardi and Rinaldo}{Nardi and
  Rinaldo}{2011}]{nardir11}
Nardi, Y. and A.~Rinaldo (2011).
\newblock Autoregressive process modeling via the lasso procedure.
\newblock {\em Journal of Multivariate Analysis\/}~{\em 102\/}(3), 528--549.


\bibitem[\protect\citeauthoryear{Rigollet and Tsybakov}{Rigollet and
  Tsybakov}{2011}]{rigollett11}
Rigollet, P. and A.~Tsybakov (2011).
\newblock Exponential screening and optimal rates of sparse estimation.
\newblock {\em The Annals of Statistics\/}~{\em 39\/}(2), 731--771.


\bibitem[\protect\citeauthoryear{Song and Bickel}{Song and
  Bickel}{2011}]{Songb11}
Song, S. and P.~Bickel (2011).
\newblock Large vector auto regressions.
\newblock {\em Arxiv preprint arXiv:1106.3915\/}.


\bibitem[\protect\citeauthoryear{Stock and Watson}{Stock and
  Watson}{2002}]{stockw02}
Stock, J. and M.~Watson (2002).
\newblock Forecasting using principal components from a large number of
  predictors.
\newblock {\em Journal of the American Statistical Association\/}~{\em
  97\/}(460), 1167--1179.


\bibitem[\protect\citeauthoryear{Stock and Watson}{Stock and
  Watson}{2006}]{stockw06}
Stock, J. and M.~Watson (2006).
\newblock Forecasting with many predictors.
\newblock In G.~Elliott, C.~W.~J. Granger, and A.~Timmermann (Eds.), {\em
  Handbook of Economic Forecasting}, Volume~1, pp.\  515--554. Elsevier,
  Amsterdam.

\bibitem[\protect\citeauthoryear{Stock and Watson}{Stock and
  Watson}{2011}]{stockw11}
Stock, J. and M.~Watson (2011).
\newblock Dynamic factor models.
\newblock In M.~Clements and D.~Hendry (Eds.), {\em Oxford Handbook of Economic
  Forecasting}, Volume~1, pp.\  35--59. Oxford University Press, Oxford.

\bibitem[\protect\citeauthoryear{Tibshirani}{Tibshirani}{1996}]{tibshirani96}
Tibshirani, R. (1996).
\newblock Regression shrinkage and selection via the lasso.
\newblock {\em Journal of the Royal Statistical Society. Series B
  (Methodological)\/}, 267--288.


\bibitem[\protect\citeauthoryear{Van De~Geer and B{\"u}hlmann}{Van De~Geer and
  B{\"u}hlmann}{2009}]{vdGB09}
Van De~Geer, S. and P.~B{\"u}hlmann (2009).
\newblock On the conditions used to prove oracle results for the lasso.
\newblock {\em Electronic Journal of Statistics\/}~{\em 3}, 1360--1392.


\bibitem[\protect\citeauthoryear{van~der Vaart and Wellner}{van~der Vaart and
  Wellner}{1996}]{vdVW96}
van~der Vaart, A.~W. and J.~A. Wellner (1996).
\newblock {\em Weak convergence and empirical processes}.
\newblock Springer Verlag.


\bibitem[\protect\citeauthoryear{Wang, Li, and Tsai}{Wang
  et~al.}{2007}]{wanglt07}
Wang, H., G.~Li, and C.~L. Tsai (2007).
\newblock Regression coefficient and autoregressive order shrinkage and
  selection via the lasso.
\newblock {\em Journal of the Royal Statistical Society: Series B (Statistical
  Methodology)\/}~{\em 69}, 63--78.


\bibitem[\protect\citeauthoryear{Yuan and Lin}{Yuan and Lin}{2006}]{yuanl06}
Yuan, M. and Y.~Lin (2006).
\newblock Model selection and estimation in regression with grouped variables.
\newblock {\em Journal of the Royal Statistical Society: Series B (Statistical
  Methodology)\/}~{\em 68\/}(1), 49--67.


\bibitem[\protect\citeauthoryear{Zhao and Yu}{Zhao and Yu}{2006}]{zhaoy06}
Zhao, P. and B.~Yu (2006).
\newblock On model selection consistency of lasso.
\newblock {\em The Journal of Machine Learning Research\/}~{\em 7}, 2541--2563.


\bibitem[\protect\citeauthoryear{Zhou, van~de Geer, and B{\"u}hlmann}{Zhou
  et~al.}{2009}]{zhou09}
Zhou, S., S.~van~de Geer, and P.~B{\"u}hlmann (2009).
\newblock Adaptive lasso for high dimensional regression and gaussian graphical
  modeling.
\newblock {\em Arxiv preprint ArXiv:0903.2515\/}.


\bibitem[\protect\citeauthoryear{Zou}{Zou}{2006}]{zou06}
Zou, H. (2006).
\newblock The adaptive lasso and its oracle properties.
\newblock {\em Journal of the American Statistical Association\/}~{\em 101},
  1418--1429.


\end{thebibliography}

\end{document}